\documentclass[11pt]{article}
\pdfoutput=1
\usepackage[margin = 1.1in,a4paper]{geometry}
\usepackage[utf8]{inputenc} 
\usepackage[T1]{fontenc} 
\usepackage{soul}
\usepackage[a4paper]{geometry} 
\usepackage{amsmath, amssymb, amsthm}
\usepackage{mathtools, mathrsfs}
\usepackage{graphicx} 
\usepackage{xcolor} 
\usepackage{authblk}
\usepackage{pgfplots}
\pgfplotsset{compat=1.18}
\usepackage[backend=biber,maxbibnames = 99, style = alphabetic, url = false, isbn = false]{biblatex}
\usepackage{hyperref}
\usepackage{xurl}
\AtEveryBibitem{\clearlist{issn}}
\AtEveryBibitem{\clearfield{pages}}
\AtEveryBibitem{\clearfield{note}}
\usepackage[page]{appendix}

\usepackage{comment}
\usepackage{enumitem}
\usepackage{multicol}
\hypersetup{
    breaklinks=true,
    colorlinks,
    citecolor={[rgb]{1,0,0}},
    filecolor={[rgb]{1,0,0}},
    linkcolor={[rgb]{1,0,0}},
    urlcolor= {[rgb]{1,0,0}}
}

\newcommand{\R}{\mathbb{R}}
\newcommand{\T}{\mathbb{T}}
\newcommand{\N}{\mathbb{N}}
\newcommand{\Z}{\mathbb{Z}}
\newcommand{\sP}{\mathscr{P}}
\newcommand{\m}{\mu}

\newcommand {\sgn}{ {\rm sgn} }
\newcommand{\co}{\colon}

\theoremstyle{plain}
\newtheorem*{thm*}{Theorem}
\newtheorem{thm}{\bf Theorem}[section]
\newtheorem{lemma}[thm]{Lemma}
\newtheorem{proposition}[thm]{Proposition}
\newtheorem{corollary}[thm]{Corollary}
\theoremstyle{remark}
\newtheorem{remark}[thm]{\bf Remark}
\newtheorem{definition}[thm]{\bf Definition}
\newcommand{\diff}{\mathop{}\!\mathrm{d}}
\numberwithin{equation}{section}

\newcommand{\fakesection}[1]{
  \par\refstepcounter{section}
  \sectionmark{#1}
  \addcontentsline{toc}{section}{\protect\numberline{\thesection}#1}
}

\title{On a Cross-Diffusion System with Independent Drifts and no Self-Diffusion: The Existence of Totally Mixed Solutions}

\author{Alp\'ar R. M\'esz\'aros\thanks{Department of Mathematical Sciences, 
Durham University,
Upper Mountjoy Campus,
Durham,
England,
DH1 3LE. 
Email: alpar.r.meszaros@durham.ac.uk} \quad
and \quad 
Guy Parker\thanks{Department of Mathematical Sciences, 
Durham University,
Upper Mountjoy Campus,
Durham,
England,
DH1 3LE. 
Email: guy.m.parker@durham.ac.uk} 
}

\date{}

\allowdisplaybreaks
\begin{document}
\maketitle

\begin{abstract}
\noindent We establish the global existence of weak solutions for a two-species cross-diffusion system, set on the 1-dimensional flat torus, in which the evolution of each species is governed by two mechanisms. 
The first of these is a diffusion which acts only on the sum of the species with a logarithmic pressure law, and the second of these is a drift term, which can differ between the two species.
Our main results hold under a \emph{total mixing} assumption on the initial data. 
This assumption, which allows the presence of vacuum, requires specific regularity properties for the ratio of the initial densities of the two species.
Moreover, these regularity properties are shown to be propagated over time. 
In proving the main existence result, we also establish the spatial $BV$ regularity of solutions. In addition, our main results naturally extend to similar systems involving reaction terms.
\end{abstract}

\vspace{.7cm}


\section{Introduction}

Cross-diffusion systems arise very naturally in mathematical biology and mathematical physics. 
Consequently, such models have received huge attention in the past half century.
It would simply be impossible to summarise the variety of these results, and hence,
for a non-exhaustive list, we refer to \cite{BurDiFPieSch:10, CheDauJun:18, CheJun:06, Jun:15} and the references therein. 
Moreover, it is often the case that such systems lie within the wider class of quasi-linear parabolic PDEs, for which a classical existence theory dates back to \cite{Ladyzhenskaya, Ama:85,Ama:86,Ama:95}.

\medskip

This manuscript concerns the existence of weak solutions to a two-species cross-diffusion system set on the $1$-dimensional flat torus, which we denote $\T := \R/ \Z$. 
Given a time horizon $T> 0$ and $V_{1},V_{2}:\T\to\R$ sufficiently smooth potentials, the model system is given by
\begin{equation}\label{eq:cdid}
\begin{cases}
    \partial_t \rho_1 = \partial_x(\rho_1\partial_x(\log(\rho_1+\rho_2)+V_1)),\\
    \partial_t \rho_2 = \partial_x(\rho_2\partial_x(\log(\rho_1+\rho_2)+V_2)),
\end{cases}
\text{on } (0,T)\times \T.
\end{equation}

In the absence of drift terms, systems resembling \eqref{eq:cdid} were initially proposed in \cite{GurtinPipkin, bertschhilhorst} as a means of modelling dispersive interacting biological species. 
In recent years, such systems have since been generalised to include reaction terms, allowing for the inclusion of growth-death mechanisms \cite{CFSS} and also, to accommodate a drift term which is common between the species \cite{JacobsRCDS}. 
Further lines of work, among others, have additionally studied the steady states of such systems in the presence of inter-species interactions \cite{burgersteadystates}.

\medskip

One of the most significant challenges in establishing the existence of solutions is imparted by the product term 
\begin{equation}\label{eq:product}
    \rho_i\partial_x\log(\rho_1+\rho_2).
\end{equation}

Indeed, given any suitable approximation of the system \eqref{eq:cdid}, one is generally only able to establish a uniform gradient estimate for the sum of the two species and not for the individual species. Whilst such estimates \emph{are} helpful in establishing the existence theory for advection-cross-diffusion systems which do not feature this product between the pressure gradient and the individual species (see, for instance, \cite{igbida2024crossdiffusion}), the term \eqref{eq:product} is problematic in the passage to the limit since it is only possible to establish the weak compactness of the terms $\rho_i$ and $\partial_x\log(\rho_1+\rho_2)$. 
As a result, it is not obvious that the product of these terms will converge to the product of their individual limits. 

\medskip
 
One could argue that the aforementioned lack of compactness is due to the lack of self-diffusion in System \eqref{eq:cdid}.  
Several works which investigate similar cross-diffusion systems (see for instance \cite{LabordeCDS,ABFS,Laurencot}) are able to circumvent this issue and obtain strong compactness for the individual species by considering an additional self-diffusion mechanism which provides sufficient gradient estimates for the individual species. That being said, it is worth remarking that the pure presence of self-diffusion mechanisms alone is in general not enough to ensure that the system is `well-behaved', if the cross-diffusion part is `too degenerate' (see for instance the very weak notion of solutions presented in \cite{DucSanYol}).

\medskip

\medskip

With a slight change of perspective, one can wonder if 
it is possible to establish higher order estimates for the sum $\rho_1+\rho_2$ (or the pressure variable, which, in our case, is $\log(\rho_1+\rho_2)$). 
Such estimates could lead to the strong convergence of the term $\partial_x\log(\rho_1+\rho_2)$ (as opposed to the strong convergence of the individual densities), and hence, allow for the passage to the limit of the product \eqref{eq:product}. 
For a different model, motivated by Aronson--B\'enilan type estimates, a breakthrough in this direction came in the work \cite{GwiPerSwi:19}. 
A similar philosophy, i.e. higher order estimates on suitable functions of the sum of the two species, has also turned out to be instrumental in various scenarios (see, for instance, \cite{BubPerPouSch:20, DavDebManSch:24, DebPerSchVau, JacobsRCDS, Jac:23, LiuXu:21, PriXu:20}).

\medskip

All these results, which rely on higher regularity estimates on the sum of the two densities, have an important philosophy in common. 
In particular, by adding up the two governing equations in these models, one obtains a `treatable' PDE for the sum, from which the necessary estimates may be derived.
For experts, it is 
evident that such an approach cannot work na\"ively in the presence of differing drift contributions (within the notation of our model, this is when $\partial_{x}V_{1}\neq\partial_{x}V_{2}$) and this is due to the spacial inhomogeneities incurred between the two species. 
Therefore, in this manuscript, we propose a totally different perspective for handling the challenge.

\medskip

To the best of our knowledge, the system closest to our model system \eqref{eq:cdid} (i.e. a system possessing only cross-diffusion and differing individual drifts) was that which was explored in \cite{MeszarosKim2018}. 
This paper studied a cross-diffusion system which was of porous medium type (instead of the logarithmic pressure law imposed in our manuscript). 
However, the main results of \cite{MeszarosKim2018} only establish the existence of solutions under an `ordering condition' on the drifts and under the assumption that the initial measures possess disjoint support. 
In particular, these conditions ensure that the densities of the two species remain segregated for all times. The segregation phenomenon and sharp interface formation for solutions was also present in the aforementioned works \cite{GurtinPipkin, bertschhilhorst,CFSS} as well as is some recent models of similar nature \cite{KimTon:21, JacKimTon:23, JacKimMes:21}.

\medskip

Looking at the structure of the terms in \eqref{eq:cdid}, it seems that there is a clear competition between the potential regularisation effect of the diffusion, acting on the sum of the two densities, and the drift terms, which act separately on the densities. 
In particular, such competition may lead to the creation of uncontrollable new interfaces between the two species.

\medskip
\medskip

{\bf The description of our results.}

\medskip

In contrast to \cite{MeszarosKim2018}, our work explores the existence of solutions under a \emph{total mixing} condition on the initial data. 
That is, we assume that, at the initial time, the ratio of the densities of the two species is bounded uniformly. 
We emphasise here that the total mixing condition does not preclude the presence of vacuum states or blow-ups in the initial data (as long as the initial data is also appropriately summable). 
\begin{figure}[h]
\caption{Suitable initial data}
\centering
\begin{tikzpicture}[scale = 0.7]
\begin{axis}[
    width=13cm,
    height=8cm,
    axis lines=middle,
    legend pos=north west,
    samples=600,
    domain=-0.6:0.6,
    restrict y to domain=0:55,
    xtick=\empty,
    y axis line style={opacity=0},
    ytick=\empty,
    enlargelimits=true
]

\addplot[blue, thick, domain=0.001:1] 
    {(x+0.5)*(1 - (x+0.5))/abs((x+0.5) - 0.5)^(1/3)*(1 - 0.2*sin(deg(17*pi*(x+0.5))))};
\addlegendentry{$\rho_{1,0}$}

\addplot[red, thick, dashed, domain=0.005:1]      
    {(x+0.5)*(1 - (x+0.5))/abs((x+0.5) - 0.5)^(1/3)*(1 - 0.3*sin(deg(11*pi*(x+0.5))))};
\addlegendentry{$\rho_{2,0}$}

\addplot[blue, thick, domain=-0.5:-0.001, forget plot] 
    {(x+0.5)*(1 - (x+0.5))/abs((x+0.5) - 0.5)^(1/3)*(1 - 0.2*sin(deg(17*pi*(x+0.5))))};

\addplot[red, thick, dashed, domain=-0.505:-0.005, forget plot] 
    {(x+0.5)*(1 - (x+0.5))/abs((x+0.5) - 0.5)^(1/3)*(1 - 0.3*sin(deg(11*pi*(x+0.5))))};

\end{axis}
\end{tikzpicture}
\label{fig:initialdata}
\end{figure}
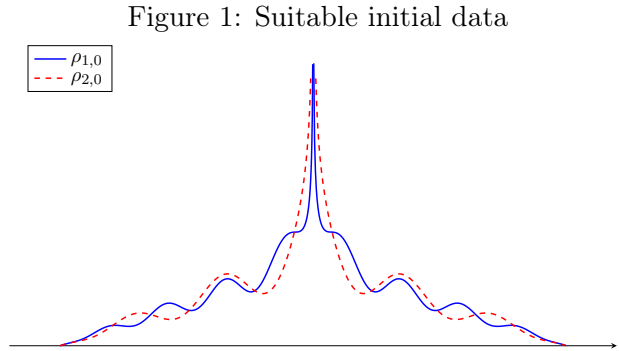
Instead, as may be observed in Figure \ref{fig:initialdata}, the total mixing condition imposes that both densities must approach the vacuum state, or points of blow-up, at a rate proportional to one another.
Moreover, since this total mixing is propagated across time, we remark that the solutions investigated in this manuscript do not form inter-species interfaces and hence, the nature of these solutions is in complete dichotomy to the segregated solutions and interfaces studied in the cross-diffusion models introduced in \cite{BertschPeletier,MeszarosKim2018,CFSS,TomPerthame}. 
Under the total mixing condition, we are able to establish existence under quite general conditions on the drift terms and without any joint structural conditions (such as ordering).

\medskip

The guiding principle behind our main result is to rewrite \eqref{eq:cdid} in terms of suitable new variables. 
Inspired by \cite{bertschhilhorst, CFSS}, it is natural to introduce the variables 
$$
\sigma:=\rho_{1} + \rho_{2} \ \ \ {\rm{and}}\ \ \ r:=\frac{\rho_{1}}{\sigma}.
$$
However, when rewritten in the variables $(\sigma, r)$, any regularisation effects present in System \eqref{eq:cdid} are still not immediately apparent due to the presence of the spacial inhomogeneities coming from the differing drifts. 

\medskip

Therefore, the first main novelty of the approach that we propose is, rather, to consider a new set of variables $(\sigma,f(r))$, where $f:(0,1)\to\R$ is a carefully chosen nonlinear function. It turns out that the nonlinearity that works well with the logarithmic pressure law is
\[
f(r) := \log\left(\frac{r}{1-r}\right) = \log\left(\frac{\rho_{1}}{\rho_{2}}\right).
\]
By making this change of variables, the new system, which is used in our analysis is
\begin{equation}\label{eq:newvar}
\begin{cases}
    \partial_t\sigma = \partial_{x}(\partial_x\sigma+ \sigma W + r\sigma V),\\[5pt]
    \partial_t f(r) =(1-r)\partial_x f(r)V + (\partial_x f(r) + V)(\partial_x\log(\sigma) +W) -VW + \partial_x V,
\end{cases}
\text{on } (0,T)\times \T,
\end{equation}
where
\[
V := \partial_x(V_1 -V_2), \ W := \partial_xV_2.
\]

While this new system \eqref{eq:newvar} could seem arbitrary, it has several deep, hidden features, which will be uncovered in our analysis. To better motivate the choice of this system, let us give more intuition on it. The second equation, for $f(r)$, can be thought of as a `more standard' transport equation. 
Indeed, if we look at the function
$$
u(t,x):= f(r(t,x)) + V_{1}(x) - V_{2}(x),
$$
direct calculation shows us that the equation for $f(r)$ can be equivalently written as 
\begin{equation}\label{eq:utransport}
\partial_{t} u(t,x) = b(t,x,u(t,x))\partial_{x}u(t,x) + g(u(t,x),x),
\end{equation}
where
$$
\begin{array}{ll}
b(t,x,u)&:=\partial_{x}\log(\sigma(t,x)) + W(x) + h(u,x)V(x),\\
g(u,x)&:= -h(u,x)V^{2}(x) -(VW)(x) + \partial_{x}V(x),\\
h(u,x)&:= \left(1+e^{u-V_{1}(x)+V_{2}(x)}\right)^{-1}. 
\end{array}
$$
Now, it is well-known for experts that, in one space dimension, with periodic boundary conditions, the transport equation \eqref{eq:utransport} preserves $BV$ estimates {\it independently} of the regularity of the drift term $b$, as long as $\partial_{u}g$ is uniformly bounded and $x\mapsto\partial_{x}g(v(x),x)$ is uniformly integrable on $\T$ for any $v:\T\to\R$ measurable. Imposing suitable regularity and summability assumptions on $V_{1},V_{2}$ as well as $BV$ regularity of $f(r(0,\cdot))$, we will be precisely in this setting.

\medskip

So, the second main novelty of our approach in this paper is to precisely follow this guiding principle, and look for the necessary $BV$ estimates by considering the time dissipation of the energy
 \begin{equation}\label{ener:intro}
  \int_{\T}|\partial_x u(t,x)| \diff x = \int_{\T}|\partial_x f(r(t,x)) +V(x)| \diff x.    
 \end{equation}  
The dissipation of this precise energy has a somewhat philosophical connection to a similar energy dissipation present for the standard Fokker--Planck equation, which played a crucial role in some new results proven using an optimal transport perspective (see \cite{santambrogiogayrat, DiMarinoSantambrogio}).

\medskip 

As this will be transparent in our analysis later in the paper, it is more suitable to work with the variables $(\sigma,f(r))$ and so, with the system \eqref{eq:newvar}, instead of the variables $(\sigma,u)$ and their corresponding system. There are several reasons for this. Firstly, suitable regularisations at the level of the original system \eqref{eq:cdid} will, by vanishing viscosity, naturally translate to the system \eqref{eq:newvar}. Secondly, and more importantly, by the choice of $f$, $BV$ estimates on the quantity $f(r)$ will translate to $BV$ estimates on the variable $r$ immediately.

\medskip 
 
Indeed, having established uniform a priori estimates on the quantity \eqref{ener:intro}, together with the crucial property of the nonlinearity $f$ that $\inf_{r\in(0,1)}f'(r)\ge 4$, we conclude with the a priori uniform $BV$ estimates for the variable $r$. 
These are then combined with more standard Sobolev estimates on $\sigma$, from which we are able to conclude the strong compactness of the two individual densities separately.

\medskip

In order to justify our computations, as hinted above, as a technical piece in our analysis, we introduce suitable regularisations for \eqref{eq:cdid}, which are compatible with the system \eqref{eq:newvar} in the new variables. 
For a vanishing parameter $\eta>0$, the regularised systems read as
\begin{equation}\label{eq:etaintro1}
        \begin{cases}
            \partial_t \rho_1 = \eta\partial_{xx}^2\rho_1 + \partial_x(\rho_1\partial_x((1-2\eta)\log(\rho_1+\rho_2)+V_1)),\\[5pt]
            \partial_t \rho_2= \eta\partial_{xx}^2\rho_2 + \partial_x(\rho_2\partial_x((1-2\eta)\log(\rho_1+\rho_2)+V_2)),
        \end{cases} \ \text{on} \ (0,T) \times \T
\end{equation}
and, correspondingly,
\begin{align}\label{eq:etaintro2}
\left\{
\begin{array}{ll}
        \partial_t\sigma &= \partial_{x}((1-\eta)\partial_x\sigma+ \sigma W + r\sigma V),\\[5pt]
           \partial_t f(r) & = \eta\partial_{xx}^2f(r) -VW + \partial_x V + (\partial_x f(r) + V)(\partial_x\log(\sigma) +W - \eta (2r-1)\partial_xf(r))
           \\[5pt]
            & + ((1-\eta)-(1-2\eta)r)\partial_x f(r) V.
\end{array}
\right.
\end{align}
It turns out that choosing the particular way to regularise the original system as in \eqref{eq:etaintro1} is important so that the system \eqref{eq:etaintro2} gains uniform parabolicity for all $\eta>0$ small. 
Subsequently, the time dissipation of 
\eqref{ener:intro} is calculated for the system \eqref{eq:etaintro2} whilst the necessary estimates are shown to be independent of $\eta$.

\medskip

Even though \eqref{eq:etaintro2} is a more standard, non-degenerate system of semi-linear parabolic PDEs, we were unable to locate classical results in the literature, precisely applicable in our setting, which would provide the desired smoothness and qualitative properties of the solutions, needed in our analysis. 
Therefore, for the convenience of the reader, we derive suitable regularity estimates and the necessary Harnack inequality for solutions of this regularised system in Section \hyperref[sec:four]{Four}.

\medskip

The main result of this paper is informally formulated as follows (for the precise statement we refer to Theorem \ref{maintheorem}).
\begin{thm}\label{thm:intro}
Suppose that $\rho_{1,0}$ and $\rho_{2,0}$ are probability densities supported on $\T$ such that 
\begin{itemize}
\item $\|\rho_{1,0}+\rho_{2,0}\|_{L\log L(\T)} := \displaystyle\int_{\T}(\rho_{1,0}+\rho_{2,0})\log (\rho_{1,0}+\rho_{2,0})\diff x<+\infty$;
\item $\log\left(\frac{\rho_{1,0}}{\rho_{2,0}}\right)\in BV(\T)$.
\end{itemize}
Suppose, moreover, that $V_{1},V_{2}:\T\to\R$ are sufficiently regular and $T>0$ is a given time horizon of arbitrary length. 
Then there exists a weak solution 
 \[
 \rho_1,\rho_2 \in L^2([0,T];BV(\T))
 \]
to system \eqref{eq:cdid} with potentials $V_1, V_2$ and initial data $\rho_{1,0},\rho_{2,0}$.
\end{thm}
To the best of our knowledge, our main theorem is the first in the literature to establish the existence of solutions for a cross-diffusion system which can handle mixed initial densities and differing drifts in the absence of self-diffusion.

\medskip

Our main contributions and results described above extend very naturally to similar systems incorporating reaction terms as well. Therefore, we conclude the manuscript by presenting the results for the corresponding cross-diffusion reaction systems, possessing the structure
\begin{equation}\label{eq:rcdid_intro}
\begin{cases}
    \partial_t \rho_1 = \partial_x(\rho_1\partial_x(\log(\rho_1+\rho_2)+V_1)) + \rho_1 F_1(\rho_1,\rho_2),\\
    \partial_t \rho_2 = \partial_x(\rho_2\partial_x(\log(\rho_1+\rho_2)+V_2)) + \rho_2 F_2(\rho_1,\rho_2),
\end{cases}
\text{on } (0,T)\times \T.
\end{equation}
Our result for this system can be informally stated as follows (we refer to Theorem \ref{reactionexistence} for the precise statement)
\begin{thm}\label{thm:main2_intro}
    Assume that there exists a suitable smooth approximation of System \eqref{eq:rcdid_intro} which admits uniformly positive densities.
    Further, assume that $\rho_{1,0}$ and $\rho_{2,0}$ admit bounded probability densities, supported on $\T$.

    \medskip
    
    If the reaction terms $F_1,F_2\co [0,+\infty)^2 \to \R$ are bounded and Lipschitz continuous and all the assumptions of Theorem \ref{thm:intro} hold, then there exists a weak solution to System \eqref{eq:rcdid_intro} with potentials $V_1,V_2$, reaction terms $F_1,F_2$ and initial data $\rho_{1,0},\rho_{2,0}$.
\end{thm}

\medskip

We end this introduction with some concluding remarks.
\begin{remark}
\begin{itemize}
\item Our model system possesses a $W_{2}$-gradient flow structure. Indeed, it is straightforward to see that System \eqref{eq:cdid} corresponds to the gradient flow of the energy
\[
E(\rho_1,\rho_2) := \int_{\T} \left[(\rho_1+\rho_2)\log(\rho_1+\rho_2) + V_1 \rho_1 + V_2\rho_2\right]\diff x.
\]
Whilst our analysis does not rely on this gradient flow structure, we use a theorem from \cite{LabordeCDS}, which establishes the well-posedness of the regularised system \eqref{eq:etaintro1} and, on the contrary, this reference does utilise the $W_{2}$-gradient flow structure.
\item This paper focuses on a particular model with logarithmic pressure law, i.e. $\log(\rho_{1}+\rho_{2})$. The assumption $\log\left(\frac{\rho_{1,0}}{\rho_{2,0}}\right)\in BV(\T)$ as well as the choice of the nonlinearity $f$ and the corresponding energy dissipation are inherent to this pressure law. 
However, we believe that our main results will also be able to shed a new perspective on similar models with different pressure laws.
\item The choice of the periodic setting of $\T$ as the state space is for purely technical reasons. 
We believe that, with suitable additional assumptions and without too many conceptual difficulties (but possibly at the price of many involved computations), the analysis of this paper could be extended to the non-compact setting of $\R$, or to bounded intervals with suitable boundary conditions. 
\item As this is not the main focus of our paper, the existence of the cross-diffusion reaction system \eqref{eq:rcdid_intro} is conditional on the existence of suitable smooth approximation. We decided to not pursue such approximations in this manuscript, however, we believe that these should follow a similar lines of thought, as for the conservative system \eqref{eq:cdid}, developed in Section \hyperref[sec:three]{Three}.
In particular, it is remarkable to notice that System \eqref{eq:rcdid_intro} rewrites neatly in the variables $(\sigma, f(r))$ as well. 
This feature is highlighted in the proof of Theorem \ref{reactionexistence}.
\end{itemize}
\end{remark}

\medskip

The structure of the rest of the paper is as follows. In Section \hyperref[sec:two]{Two}, we introduce the relevant preliminaries for this manuscript and present the main result, Theorem \ref{maintheorem}, along with its accompanying hypotheses. 
In Section \hyperref[sec:three]{Three}, we introduce a suitable approximation scheme for System \eqref{eq:cdid} before establishing the main $BV$ estimate for $f(r)$ and proving the main theorem.
In Section \hyperref[sec:four]{Four}, we establish the smoothness of the approximating system introduced in Section \hyperref[sec:three]{Three}, thus, justifying the calculations of the previous section. Finally, in Section \hyperref[sec:five]{Five} we demonstrate how our analysis can be extended to systems involving reaction terms as well. We end the manuscript with an appendix section, where we have collected a few  technical results.

\section{Preliminaries and Main Theorem}\label{sec:two}
\subsection{Function and Measure Spaces}
Given $u \in L^1(\T)$, its total variation $TV(u)$ is defined
\[
TV(u):=\sup \bigg\{\int_{\T} u \partial_x \varphi \diff x \ \bigg| \ \varphi \in C^1(\T), \ \|\varphi\|_{L^{\infty}}<1 \bigg\}.
\]
Subsequently, we say that $u$ is of bounded variation if $TV(u) < +\infty$. 
Moreover, we define the space of functions of bounded variation by
\[
BV(\T):=\{u \in L^1(\T) \ | \ TV(u) < +\infty\}.
\]
When equipped with the norm 
\[\|u\|_{BV(\T)}:= \|u\|_{L^1(\T)} + TV(u),\]
$BV(\T)$ defines a Banach space.
Moreover, if $u \in W^{1,1}(\T)$, then $u$ satisfies the equality 
\[
\|u\|_{BV(\T)} = \|u\|_{W^{1,1}(\T)}.
\]

We let $\mathscr{M}_{+}(\T)$ denote the space of non-negative finite Radon measures on $\T$ and define the space of probability measures 
\[
\sP(\T) := \{\m \in \mathscr{M}_{+}(\T) \ | \ \m(\T) =1\}.
\]
Moreover, throughout this paper, we consider $\sP(\T)$ endowed with the topology induced by the weak-$*$ convergence of measures, that is, the topology induced by duality with $C(\T)$.
\subsection{Assumptions and Main Theorem}
\begin{definition}\label{mainsolution}
Given Lipschitz continuous potentials $V_1, V_2\co \T \to \R$ and initial data $\rho_{1,0},\rho_{2,0}\in \sP(\T)\cap L^1(\T)$, we say that the pair $(\rho_1,\rho_2)$ is a weak solution to System \eqref{eq:cdid} if: 
    \begin{enumerate}
        \item $\rho_1,\rho_2 \in C([0,T];\sP(\T))$,
        \item $\rho_1 + \rho_2 \in L^1([0,T];W^{1,1}(\T))$.
    \end{enumerate}
    
    In addition, we demand that the following equality is satisfied for any $\varphi \in C^\infty_c([0,T)\times \T)$ and $i \in \{1,2\}$.
    \begin{equation}\label{eq:weakcdid}
        \int_{\T} \rho_{i,0}\varphi_0 \diff x + \int_0^T\int_{\T}\rho_i\partial_t\varphi \diff x \diff t = \int_0^T\int_{\T} \rho_i\partial_x(\log(\rho_1+\rho_2)+V_i)\partial_x \varphi\diff x \diff t.
    \end{equation}
    \end{definition}
    
To guarantee the existence of solutions, the following hypotheses on the initial data and drift inducing potentials are made.
\begin{multicols}{2}
\begin{enumerate}[label=$(\mathbf{H \arabic*})$, ref = $\mathbf{H \arabic*}$]
    \item \label{H1}
    $\|\rho_{1,0}+\rho_{2,0}\|_{L\log L(\T)} <+\infty$;
    \item \label{H2}
    $\log\left(\frac{\rho_{1,0}}{\rho_{2,0}}\right)\in BV(\T)$;
    \item \label{H3}
    $V_1,V_2 \in W^{2,1}(\T)$;
    \item \label{H4}
    $V_1-V_2 \in W^{3,1}(\T)$.
\end{enumerate}
\end{multicols}

\begin{thm}\label{maintheorem}
    If the hypotheses \ref{H1}-\ref{H4} hold, then there exists a weak solution $(\rho_1,\rho_2)$ to System \eqref{eq:cdid} with potentials $V_1, V_2$ and initial data $(\rho_{1,0},\rho_{2,0})$.
    Moreover, this solution possesses the regularity:
    \[\rho_1,\rho_2 \in L^2([0,T];BV(\T)) \cap L^p_{loc}((0,T];BV(\T))\]
    for every $p \in [1,\infty)$.
\end{thm}

\begin{remark}\label{mixing}
\begin{itemize}
    \item  We once again emphasise that Hypothesis \ref{H2} does not preclude the presence of the vacuum state or blow-ups in the initial datum.
    On the other hand, in enforcing Hypothesis \ref{H2}, it follows that points of vacuum and points of blow-up in each species must coincide. 
    Moreover, the two initial densities must become asymptotically proportional to one another as a state of vacuum or state of blow-up is approached. 
    \item
    Since we are in a $1$-dimensional setting, the continuous embedding $BV(\T)\hookrightarrow L^\infty(\T)$ holds. 
    Consequently, Hypothesis \ref{H2} further implies that the ratio $\frac{\rho_{1,0}}{\rho_{2,0}}$ and its reciprocal are both bounded uniformly on $\T$. 
    It is this condition in particular that we refer to as \emph{total mixing}.
    \item We would like to underline here that the imposed summability assumptions on the initial data are minimal: we only impose $\rho_{1,0},\rho_{2,0}$ to be probability densities, whose sum is in $L\log L (\T)$. If the initial data would satisfy higher summability, i.e. $\rho_{1,0},\rho_{2,0} \in L^q(\T)$, for $q >2$, then the solution we construct would also possess the regularity $L^q([0,T];BV(\T))$.
\end{itemize}
   
\end{remark}

\section{$BV$ Estimates and Existence}\label{sec:three}
Before establishing any a priori estimates for System \eqref{eq:cdid}, it is first necessary to approximate System \eqref{eq:cdid} with a family of systems for which the well-posedness theory is on surer footing. 
To this end, we defer to the existence theory previously established in \cite[Theorem 2.2]{LabordeCDS}. 
Moreover, we consider a sequence of regularised systems, each similar to System \eqref{eq:cdid} but featuring an additional linear self-diffusion term in the equation governing each species - i.e. for $\eta\in(0,\frac{1}{2}]$, we consider systems of the form 
\begin{equation}\label{eq:regcdid}
        \begin{cases}
            \partial_t \rho_1 = \eta\partial_{xx}^2\rho_1 + \partial_x(\rho_1\partial_x((1-2\eta)\log(\rho_1+\rho_2)+V_1)),\\
            \partial_t \rho_2= \eta\partial_{xx}^2\rho_2 + \partial_x(\rho_2\partial_x((1-2\eta)\log(\rho_1+\rho_2)+V_2)),
        \end{cases} \ \text{on} \ (0,T) \times \T.
\end{equation}
We refer to System \eqref{eq:regcdid} as the \emph{$\eta$-Regularised System} and, in the following, we define the relevant notion of weak solution for System \eqref{eq:regcdid} and recall from \cite[Theorem 2.2]{LabordeCDS} its accompanying existence result.

\begin{definition}
Let $\eta \in (0,\frac{1}{2}]$. 
Given Lipschitz continuous potentials $V_1, V_2\co \T \to \R$ and initial data $\rho_{1,0},\rho_{2,0}\in \sP(\T)\cap L^1(\T)$, a pair $(\rho_1,\rho_2)$ is a weak solution to the \emph{$\eta$-Regularised System} \eqref{eq:regcdid} if 
\[
\rho_1,\rho_2\in C([0,T];\sP(\T))\cap L^1([0,T];W^{1,1}(\T))
\]
and, if the following equality is satisfied for any $\varphi \in C_{c}^\infty([0,T)\times \T)$ and $i \in \{1,2\}$.
    \begin{equation}\label{eq:weaketacdid}
    \begin{split}
        \int_{\T} \rho_{i,0}\varphi_0 \diff x &  + \int_0^T\int_{\T}\rho_i\partial_t\varphi \diff x \diff t 
        \\
        & = \int_0^T\int_{\T} \rho_i\partial_x((1-2\eta)\log(\rho_1+\rho_2)+V_i)\partial_x \varphi + \eta \partial_x\rho_i \partial_x\varphi \diff x \diff t.
    \end{split}
    \end{equation}
\end{definition}

\begin{thm}{\cite[Theorem 2.2]{LabordeCDS}}
If $\rho_{1,0},\rho_{2,0}$ satisfy Hypothesis \ref{H1} then there exists a weak solution to the $\eta$-Regularised System \eqref{eq:regcdid} with initial data $(\rho_{1,0},\rho_{2,0})$.
\end{thm}
\begin{proof}
    Up to the constants $\eta$ and $(1-2\eta)$, the proof of this result follows directly from \cite[Theorem 2.2]{LabordeCDS} in the case `$m=1$'. 
    In particular, the Theorem of Laborde postulates the existence of solutions under the assumption that the associated energy is finite. 
    In this case, the associated energy functional is given by 
    \[
    E_\eta(\rho_1,\rho_2) := \int_{\T} [(1-2\eta)(\rho_1+\rho_2)\log(\rho_1+\rho_2) + \eta \rho_1\log(\rho_1)+ \eta \rho_2\log(\rho_2) + V_1 \rho_1 + V_2\rho_2 ]\diff x.
    \]
    Moreover, since
    \[
    x\log(x) + y\log(y) \leqslant (x+y)\log(x+y) +1
    \]
    for all $x,y \geqslant 0$, the requirement $E_{\eta}(\rho_{1,0},\rho_{2,0})< + \infty$ is satisfied whenever the initial data satisfies Hypothesis \ref{H1} and the potentials satisfy $V_1, V_2\in L^\infty(\T)$. 
    Since we assume that $V_1,V_2 \in W^{1,\infty}(\T)$ this is indeed the case, and hence, it is simply left to consider the adaptation of the result of Laborde from a bounded interval with no-flux boundary conditions to the case of the $1$-dimensional flat torus.
\end{proof}

Since the additional self-diffusion terms in System \eqref{eq:regcdid} provide a regularising effect, a classical bootstrapping argument (which we defer to Section \hyperref[sec:four]{Four}) shows that, under sufficient regularity assumptions on the potentials and initial data, weak solutions to System \eqref{eq:regcdid} are, in fact, smooth and uniformly positive. 

\medskip

Using this additional smoothness, we calculate the time dissipation for the energy \eqref{ener:intro} and, subsequently, bound the quantity $f(r)$ in the space $L^\infty([0,T];BV(\T))$ via a Gr\"onwall type argument. 
In particular, this bound will not depend on the choice of $\eta$.
Moreover, in establishing this $BV$ estimate, we recall that the variables $(\sigma, r)$ denote the change of co-ordinates given by 
\[
\sigma:=\rho_{1} + \rho_{2} \ \ \ {\rm{and}}\ \ \ r:=\frac{\rho_{1}}{\sigma}.
\]
Meanwhile, the non-linearity $f(r)$ relates to the logarithmic ratio between the two species via the following equality.

\[
f(r) := \log\bigg(\frac{r}{1-r}\bigg) = \log\bigg(\frac{\rho_1}{\rho_2}\bigg).
\]

Motivated by the above equality, we perform the transformation of variables $(\rho_1,\rho_2) \mapsto (\sigma,f(r))$ and, further, make the change of potentials:
\begin{equation}\label{def:VW}
V := \partial_x(V_1 -V_2), \ W := \partial_xV_2.
\end{equation}

\begin{remark}
\begin{itemize}
\item For the remainder of this section, we tacitly assume that $V,W \in C^\infty(\T)$ and, further, assume the existence of smooth and uniformly positive solutions to System \eqref{eq:weaketacdid}. 
Such assumptions are fully justified in Section \hyperref[sec:four]{Four}.
\item Further, we clarify that, when a function or solution $g\co D \to \R$ is referred to as uniformly positive, we mean \emph{strict uniform positivity} in the sense that there exists $c>0$ such that $g(x) > c > 0$ for all $x\in D$.
\end{itemize} 
\end{remark}
\begin{proposition}\label{changevar}
       Let $(\rho_1,\rho_2)$ denote a smooth solution to the $\eta$-Regularised system \eqref{eq:regcdid}. 
       The variable $\sigma$ satisfies the equality
       \begin{equation}\label{eq:sigmaeq}
        \partial_t\sigma = \partial_{x}((1-\eta)\partial_x\sigma+ \sigma W + r\sigma V) \text{ on } (0,T)\times \T.
        \end{equation}
        If, in addition, $\rho_1,\rho_2$ admit uniformly positive densities on $[0,T]\times \T$, then the variable $f(r)$ satisfies the equality
        \begin{equation}\label{eq:f(r)}
        \begin{split}
           \partial_t f(r) & = \eta\partial_{xx}^2f(r) -VW + \partial_x V  \\
           & + (\partial_x f(r) + V)\bigg(\partial_x\log(\sigma) +W - \eta (2r-1)\partial_xf(r)\bigg)
           \\
            & + ((1-\eta)-(1-2\eta)r)\partial_x f(r) V.
        \end{split}
        \end{equation}
\end{proposition}

\begin{proof} 
    Since the densities $\rho_1,\rho_2$ are assumed to be smooth, System \eqref{eq:regcdid} is satisfied pointwise everywhere on $(0,T)\times \T$.
    Moreover, in summing System \eqref{eq:regcdid} over $i\in \{1,2\}$, we obtain the equation governing $\sigma$. 
    \begin{flalign*}
        \partial_t\sigma & = (1-\eta) \partial_{xx}^2(\rho_1+\rho_2) + \partial_x((\rho_1+\rho_2)\partial_xV_2+ \rho_1\partial_x(V_1-V_2)) 
        \\
        & = (1-\eta)\partial^2_{xx}\sigma + \partial_x(\sigma W + r\sigma V).
    \end{flalign*}
    Further, assuming that $\rho_1,\rho_2$ are uniformly positive, we apply the product rule to the quotient $\rho_1\sigma^{-1}$. 
    From this, we deduce the following equality for $\partial_t r$.
    \[
    \partial_t r = \frac{1}{\rho_1 + \rho_2}\partial_t \rho_1 - \partial_t(\rho_1+\rho_2)\frac{\rho_1}{(\rho_1+\rho_2)^2} = \frac{1}{\sigma}\partial_t\rho_1- \frac{r}{\sigma}\partial_t\sigma
    \]
    From System \eqref{eq:regcdid}, it follows that
    \begin{flalign*}
        \frac{1}{\sigma}\partial_t\rho_1 & = \frac{1}{\sigma}\partial_x(r((1-2\eta)\sigma\partial_x\log(\sigma)+\sigma W) + r\sigma V + \eta \partial_x(r\sigma)) 
        \\
        & = \frac{r}{\sigma}\partial_x((1-2\eta)\sigma\partial_x\log(\sigma) + \sigma W) + \eta \frac{r}{\sigma}\partial_{xx}^2\sigma + \frac{1}{\sigma}\partial_x(r\sigma V) \\
        & + \eta \partial_{xx}^2r +  2\eta\partial_{x}r\partial_x\log(\sigma) 
        + \partial_x r((1-2\eta)\partial_x\log(\sigma)+W).
        \end{flalign*}
        Similarly, from Equation \eqref{eq:sigmaeq}, it follows that
        \begin{flalign*}
        \frac{r}{\sigma}\partial_t\sigma & =  \frac{r}{\sigma}\partial_x((1-2\eta)\sigma\partial_x\log(\sigma) + \sigma W) + \eta \frac{r}{\sigma}\partial_{xx}^2\sigma + \frac{r}{\sigma}\partial_x(r\sigma V).
     \end{flalign*}
    Consequently, $\partial_t r$ satisfies
    \begin{flalign*}
    \partial_t r
    & =\partial_x r((1-2\eta)\partial_x\log(\sigma)+W) + \frac{1-r}{\sigma}\partial_x(r\sigma V) + \eta \partial_{xx}^2r + 2\eta\partial_xr\partial_x\log(\sigma)
    \\
    & =  \eta\partial_{xx}^2r + (\partial_x\log(\sigma) +W)(\partial_x r + r(1-r)V) -r(1-r)VW + (1-r)\partial_x(rV).
    \end{flalign*}
Finally, to produce the equation governing $\partial_tf(r)$ we recognise that 
$f'(r) = (r(1-r))^{-1}$.
Moreover, since $\rho_1$ and $\rho_2$ are assumed to be uniformly positive, it follows that $r$ is bounded away from the boundary of the interval $(0,1)$ and hence, $f'(r)$ is smooth on $(0,T)\times \T$.
Subsequently, we are justified in applying the chain rule to the function $\partial_t f(r)$. 
From this application of the chain rule, we deduce the following equality.
\begin{equation}\label{eq:intermediate1}
\begin{split}
    \partial_tf(r) & = f'(r)\partial_tr
    \\
    & = (f'(r)\partial_x r + V)(\partial_x\log(\sigma) +W)- VW +\frac{1}{r}\partial_x(rV)  + \eta f'(r)\partial_{xx}^2r
    \\
    & = (\partial_x f(r) + V)(\partial_x\log(\sigma) +W) - \eta(\partial_x r)^2 f''(r) 
    \\
    & - VW +(1-r)\partial_xf(r)V + \partial_x V  + \eta \partial_{xx}^2f(r) .
\end{split}
\end{equation}
Since $f''(r) = (f'(r))^2(2r-1)$, it follows that
\begin{equation}\label{eq:intermediate2}
\begin{split}
(\partial_xf(r) & + V)(\partial_x\log(\sigma) +W) - \eta(\partial_x r)^2 f''(r) 
\\
& =  (\partial_xf(r) + V)\bigg(\partial_x\log(\sigma) +W - \eta(2r-1)\partial_xf(r)\bigg)+\eta(2r-1)\partial_x f(r)V.
\end{split}
\end{equation}
By substituting Equality \eqref{eq:intermediate2} into Equation \eqref{eq:intermediate1}, we conclude that Equation \eqref{eq:f(r)} holds.
\end{proof}

\subsection{The $BV$ Estimates}
\begin{thm}\label{mainestimate}
        Let $(\rho_1,\rho_2)$ denote a solution to System \eqref{eq:regcdid} for which $\rho_1,\rho_2$ admit smooth and uniformly positive densities. 
        Further, let $(\sigma,f(r))$ denote the corresponding smooth solutions to Equations \eqref{eq:sigmaeq}, \eqref{eq:f(r)} with initial data $(\sigma_0,f(r_0))$. 
        There exists a constant $C \in \R$, which depends only on $T, W, V$ and $TV(f(r_0))$, such that 
        \[
        \|\partial_x f(r)\|_{L^\infty([0,T];L^1(\T))}< C.
        \] 
    \end{thm}
\begin{proof}   
    To establish the result, we consider the dissipation of the first order energy given by 
    \begin{equation}\label{eq:firstenergy}
    \int_{\T}|\partial_x f(r) +V| \diff x.    
    \end{equation} 
    Since we assume that $f(r)$ and $\sigma$ are assumed to be smooth on $[0,T]\times \T$ and since this set is compact, we the emphasise that the energy \eqref{eq:firstenergy} and, more generally, any Lipschitz continuous functions which may be composed with $f(r)$ or $\sigma$ (or their respective derivatives) will always be summable on $[0,T]\times \T$. 
    Moreover, the following energy dissipation equality is justified as a consequence of Lemma \ref{gradientinterchangelemma}.
    \begin{equation}\label{eq:ede}
    \begin{split}
    \int_{\T} |\partial_xf(r_s)+ V| \diff x & = \int_{\T} |\partial_xf(r_0)+ V |  \diff x 
     \\ & + \int_0^s\int_{\T} \sgn(\partial_xf(r_t)+ V)\partial_t\partial_xf(r_t) \diff x \diff t.
    \end{split}
    \end{equation}
    Ultimately, we would like to apply a Gr\"onwall argument to the energy \eqref{eq:firstenergy}, hence, it remains to bound the right hand side of Equation \eqref{eq:ede} in terms of \eqref{eq:firstenergy}.
    We begin by plugging Equation \eqref{eq:f(r)} into the righter most term of Equation \eqref{eq:ede} acknowledging that, since $f(r)$ is assumed to be smooth, the order of partial differentiation can be swapped.
    \begin{equation}\label{eq:firstexpression}
    \begin{split}
        \int_0^s & \int_{\T} \sgn(\partial_xf(r)+ V )\partial_x\partial_tf(r) \diff x \diff t = \int_0^s\int_{\T} \eta \partial_{xxx}^3f(r) \sgn(\partial_xf(r) +V) \diff x \diff t
        \\
        + &\int_0^s\int_{\T}\sgn(\partial_xf(r) +V)\partial_x\bigg((\partial_x\log(\sigma) +W - \eta (2r-1)\partial_xf(r))(\partial_x f(r) + V)\bigg) \diff x \diff t 
        \\
        + & \int_0^s\int_{\T}\sgn(\partial_xf(r) +V)( \partial_{xx}^2 V -\partial_x(VW)) \diff x \diff t
        \\
        + & \int_0^s\int_{\T}\sgn(\partial_xf(r) +V)\partial_x\bigg(((1-\eta)-(1-2\eta)r)\partial_x f(r)V\bigg) \diff x \diff t =: I_1.
    \end{split}
    \end{equation}
    We first address the fourth and second lines of Equation \eqref{eq:firstexpression}.
    In particular, since the term $((1-\eta)-(1-2\eta)r)$ is positive for $\eta \in (0,\frac{1}{2}], r \in [0,1]$, it then follows that
    \begin{equation}\label{eq:rearrangement}
    \begin{split}
    & \sgn(\partial_xf(r) +V)\partial_x\bigg(((1-\eta)-(1-2\eta)r)\partial_x f(r)V\bigg) 
    \\
    & =  \sgn\bigg(((1-\eta)-(1-2\eta)r)(\partial_x f(r)+V)\bigg)\partial_x\bigg(((1-\eta)-(1-2\eta)r)(\partial_x f(r)+V)V\bigg) 
    \\
    & -  \sgn\bigg(((1-\eta)-(1-2\eta)r)(\partial_x f(r)+V)\bigg)\partial_x\bigg(((1-\eta)-(1-2\eta)r)V^2\bigg).
    \end{split}
    \end{equation}
    By applying Lemma \ref{modulusinterchangelemma} to Equation \eqref{eq:rearrangement} and to the expression given in the second line of Equation \eqref{eq:firstexpression}, we derive the following equalities which hold almost everywhere on $[0,T]\times \T$.
    \begin{align}\label{eq:modsub1}
    \begin{split}
    \sgn(\partial_xf(r) +V)& \partial_x\bigg(((1-\eta)-(1-2\eta)r)\partial_x f(r)V\bigg) 
    \\
    = & \partial_x\bigg(((1-\eta)-(1-2\eta)r)(|\partial_x f(r)+V|)V\bigg)
    \\
    - & \sgn(\partial_x f(r)+V)\partial_x\bigg(((1-\eta)-(1-2\eta)r)V^2\bigg),
    \end{split}
    \\
     \begin{split}\label{eq:modsub2}
         \sgn(\partial_xf(r) +V)&\partial_x\bigg((\partial_x\log(\sigma) +W - \eta (2r-1)\partial_xf(r))(\partial_x f(r) + V)\bigg) 
         \\
         = 
         &\partial_x\bigg((\partial_x\log(\sigma) +W - \eta (2r-1)\partial_xf(r))(|\partial_x f(r) + V|)\bigg).
     \end{split}
    \end{align}
    Substituting Equations \eqref{eq:modsub1} and \eqref{eq:modsub2} into the fourth and second lines of Equation \eqref{eq:firstexpression} respectively, the subsequent equality for $I_1$ follows.
    \begin{equation}\label{eq:again}
    \begin{split}
        I_1 & = \int_0^s\int_{\T} \eta \partial_{xx}^2(\partial_xf(r) +V)\sgn(\partial_xf(r) +V) \diff x \diff t
        \\
        & + \int_0^s\int_{\T}\partial_x\bigg((\partial_x\log(\sigma) +W - \eta (2r-1)\partial_xf(r))(|\partial_x f(r) + V|)\bigg) \diff x \diff t 
        \\
         & + \int_0^s\int_{\T}\partial_x\bigg(((1-\eta)-(1-2\eta)r)V(|\partial_x f(r)+V|)\bigg) \diff x \diff t
         \\
        &+  \int_0^s\int_{\T}\sgn(\partial_xf(r) +V)( (1-\eta)\partial_{xx}^2 V -\partial_x(VW)) \diff x \diff t
        \\
         & - \int_0^s\int_{\T}\sgn(\partial_xf(r) +V)\partial_x\bigg(((1-\eta)-(1-2\eta)r)V^2\bigg)  \diff x \diff t.
         \end{split}
   \end{equation}
To address the first term in Equation \eqref{eq:again}, we recognise that, since $f(r)$ is assumed to be smooth, the inequality
\begin{equation}\label{eq:katoapplied}
\sgn(\partial_xf(r) + V)\partial_{xx}^2(\partial_xf(r) + V) \leqslant \partial_{xx}^2(|\partial_xf(r) +V|) 
\end{equation}
is satisfied in the sense of distributions. 
Inequality \eqref{eq:katoapplied} follows from Kato's Inequality (see Lemma \ref{eq:Kato}).
From Kato's result, it also follows that 
\[
\partial_{xx}^2(|\partial_xf(r_t) + V|) \in \mathscr{M}(\T)
\]
for almost every $t \in [0,T]$.
Moreover, from Inequality \eqref{eq:katoapplied}, it follows that
\begin{equation}\label{eq:andagain}
\begin{split}
       I_1 & \leqslant \int_0^s\int_{\T} \partial_{xx}^2( \eta|\partial_xf(r) +V|)\diff x \diff t
       \\
        & + \int_0^s\int_{\T}\partial_x\bigg((\partial_x\log(\sigma) +W - \eta (2r-1)\partial_xf(r))(|\partial_x f(r) + V|)\bigg) \diff x \diff t 
        \\
        & + \int_0^s\int_{\T}\partial_x\bigg(((1-\eta)-(1-2\eta)r)V(|\partial_x f(r)+V|)\bigg) \diff x \diff t
        \\
         &+  \int_0^s\int_{\T}\sgn(\partial_xf(r) +V)( (1-\eta)\partial_{xx}^2 V -\partial_x(VW)) \diff x \diff t
        \\
         & - \int_0^s\int_{\T}\sgn(\partial_xf(r) +V)\partial_x\bigg(((1-\eta)-(1-2\eta)r)V^2\bigg) \diff x \diff t.    
\end{split}
\end{equation}
Now that the first three terms of Equation \eqref{eq:andagain} are expressed as integrals of spatial gradients, we may use that $\T$ is without boundary to simplify this upper bound for $I_1$.
In particular, we recognise that, as a consequence Stokes' Theorem, the equality
\begin{equation}\label{eq:stokes}
    \int_{\T} \partial_x g \diff x = 0
\end{equation}
holds.
In particular, Equation \eqref{eq:stokes} holds for all $g\in BV(\T)$  (see \cite[Theorem 5.1]{evansgariepy}) and hence, Equation \eqref{eq:stokes} applies to the first (and second and third) line of Equation \eqref{eq:andagain}, even when $\partial_{xx}^2(|\partial_xf(r) + V|)$ is only a measure.
Consequently, the first three terms on the right-hand side of Inequality \eqref{eq:andagain} vanish and the following inequality is thus satisfied.
\begin{flalign*}
       I_1 & \leqslant \underbrace{\int_0^s\int_{\T}\sgn(\partial_xf(r) +V)( (1-\eta)\partial_{xx}^2 V -\partial_x(VW)) \diff x \diff t}_{=:I_2}\\
         & - \underbrace{\int_0^s\int_{\T}\sgn(\partial_xf(r) +V)\partial_x\bigg(((1-\eta)-(1-2\eta)r)V^2\bigg) \diff x \diff t}_{=:I_3}.
\end{flalign*}
It is immediately clear from H\"older's inequality that 
\begin{equation}\label{eq:IIbound}
    I_2 \leqslant s(\|V\|_{W^{2,1}(\T)} + \|VW\|_{W^{1,1}(\T)}),
\end{equation}
hence, we are left to bound $I_3$.
In particular, 
\[
((1-\eta)-(1-2\eta)r), \ (1-2\eta)\in [0,1] \ \text{for} \ r \in [0,1], 
 \ \eta \in \bigg(0,\frac{1}{2}\bigg]
\]
and hence, expanding the derivative in $I_3$, Inequality \eqref{eq:IIIbound} follows.
\begin{equation}\label{eq:IIIbound}
    I_3 \leqslant s\|V^2\|_{W^{1,1}(\T)} + \|V^2\|_{L^\infty(\T)}\|\partial_x r\|_{L^1([0,s]\times\T)}.
\end{equation}
Now, using that $f'(r) \geqslant 4$ for any $r\in [0,1]$,  the following upper bound for $I_3$ holds by use of Equation \eqref{eq:IIIbound}, in conjunction with the triangle inequality.
\begin{equation}\label{eq:IIIbound2}
\begin{split}
    I_3 -  s\|V^2\|_{W^{1,1}(\T)}& \leqslant  \|V^2\|_{L^\infty(\T)}\|\partial_x r\|_{L^1([0,s]\times\T)} 
    \\
    &\leqslant \|V^2\|_{L^\infty(\T)}(\inf_\tau f'(\tau))\|\partial_x r\|_{L^1([0,s]\times\T)}
    \\
    & \leqslant  \|V^2\|_{L^\infty(\T)}\|\partial_x f(r)\|_{L^1([0,s]\times\T)} 
    \\
    & \leqslant  \|V^2\|_{L^\infty(\T)}(\|\partial_x f(r)+V\|_{L^1([0,s]\times\T)} + \|V\|_{L^1([0,s]\times\T)}).
\end{split}
\end{equation}
Ultimately, we deduce the desired energy dissipation inequality by combining Inequalities \eqref{eq:IIbound} and \eqref{eq:IIIbound2}, which together bound $I_1$, with the Energy Dissipation Equality \eqref{eq:ede}.
The inequality reads as follows:
\begin{equation}\label{eq:EDI}
\begin{split}
   & \int_{\T} |\partial_x f(r_s)+ V| \diff x \leqslant \int_{\T} |\partial_x f(r_0)+ V|  \diff x  + \|V^2\|_{L^\infty(\T)}\int_0^s\int_{\T} |\partial_x f(r_t) + V| \diff x \diff t \\
   & + s(\|V\|_{W^{2,1}(\T)} + \|VW\|_{W^{1,1}(\T)}+ \|V^2\|_{W^{1,1}(\T)} + \|V^2\|_{L^\infty(\T)}\|V\|_{L^1(\T)}).
\end{split}
\end{equation}
For notational convenience, let
\begin{flalign*}
    \alpha(V,W)& :=\|V\|_{W^{2,1}(\T)} + \|VW\|_{W^{1,1}(\T)}+ \|V^2\|_{W^{1,1}(\T)} + \|V^2\|_{L^\infty(\T)}\|V\|_{L^1(\T)},\\
    \beta(V) & := \|V^2\|_{L^\infty(\T)}.
\end{flalign*}
The desired estimate now follows for any $s\in [0,T]$.
In particular, we derive this inequality by using the integral form of Gr\"onwall's inequality in conjunction with the Energy Dissipation Inequality \eqref{eq:EDI}. 
\begin{equation}\label{eq:Gronwall}
    \int_{\T} |\partial_x f(r_t)+ V| \diff x\bigg|_{t=s}\leqslant \bigg(\int_{\T} |\partial_x f(r_0)+ V|  \diff x + \alpha s\bigg)\exp(s\beta).
\end{equation} 
When $\partial_x f(r_0) \in L^1({\T})$ and the potentials $V,W$ are smooth, the left hand side of Inequality \eqref{eq:Gronwall} is finite and, consequently, it follows from Inequality \eqref{eq:Gronwall} that $\partial_xf(r) \in L^\infty([0,T];L^1(\T))$ with the bound on this quantity depending only on $T, V, W$ and $r_0$ through the constants $\alpha,\beta$ and through $\|\partial_x f(r_0)\|_{L^1(\T)}$.
\end{proof}
\begin{thm}\label{sumestimate}
Let $(\rho_1,\rho_2)$ be a weak solution of System \eqref{eq:regcdid} with initial data $(\rho_{1,0},\rho_{2,0})$. 
If the initial data satisfies Hypothesis \ref{H1}
then there exists a constant $C\in \R$ which depends only on $T, \|V_i\|_{W^{1,\infty}(\T)}$ and $\|\rho_{1,0}+\rho_{2,0}\|_{L\log L(\T)}$ such that the inequalities
\begin{flalign*}
  & \|\partial_x(\rho_1+\rho_2)\|_{L^\frac{3}{2}([0,T]\times\T)} < C,\\
  & \|\rho_1+\rho_2\|_{L^2([0,T];BV(\T))} < C  
\end{flalign*}
are satisfied.
\end{thm}
\begin{proof}
    Following the proof of \cite[Proposition 3.4]{LabordeCDS}, one obtains the existence of a constant $A$ which depends only $T,\|V_i\|_{W^{1,\infty}(\T)}$ and $\|\rho_{1,0}+\rho_{2,0}\|_{L\log L(\T)}$ such that
    \[
    \|\sqrt{\rho_1+\rho_2}\|_{L^2([0,T];H^1(\T))} < A.
    \]
    
    A standard embedding result (see \cite[page 74, chapter II, $\mathsection 3$]{Ladyzhenskaya}) tells us that 
      \[
        L^2([0,T];H^1(\T)) \cap L^\infty([0,T];L^2(\T)) \hookrightarrow L^6([0,T]\times\T).
      \]
    It then follows that $\sqrt{\rho_1^{\tau}+ \rho_2^\tau}$ is bounded uniformly in $L^6([0,T]\times\T)$ and, hence, as a consequence of H\"older's inequality, the product 
      \[
      \partial_x(\rho_1+ \rho_2) = 2\partial_x(\sqrt{\rho_1+ \rho_2})\sqrt{\rho_1+ \rho_2} 
      \]
    must be bounded in $L^{\frac{3}{2}}([0,T]\times\T)$ by a constant $B$ depending only on $A$ and $T$.
    Finally, since $\rho_i$ are curves of probability measures, we apply H\"older's inequality to deduce the inequality
      \begin{flalign*}
      \|\rho_1+\rho_2\|_{L^2([0,T];BV(\T))} & = \|\rho_1+\rho_2\|_{L^2([0,T];L^1(\T))}+\|\partial_x(\rho_1+\rho_2)\|_{L^2([0,T];L^1(\T))} \\
      & = \|\rho_1+\rho_2\|_{L^2([0,T];L^1(\T))}+\|2\partial_x(\sqrt{\rho_1+ \rho_2})\sqrt{\rho_1+ \rho_2} \|_{L^2([0,T];L^1(\T))} \\
      & \leqslant
      2T^\frac{1}{2} + 2\|\partial_x\sqrt{\rho_1+\rho_2}\|_{L^2([0,T];L^2(\T))}\|\sqrt{\rho_1+\rho_2}\|_{L^\infty([0,T];L^2(\T))}
      \\
      & \leqslant
      2T^\frac{1}{2} + 4A.
      \end{flalign*}
      We conclude the proof by taking $C = \max\{ 2T^\frac{1}{2} + 4A, B\}$.
\end{proof}

    \begin{corollary}\label{maincorollary}
         Let $(\rho_1,\rho_2)$ be a smooth solution to the $\eta$-Regularised System \eqref{eq:regcdid}. 
         If $\rho_1,\rho_2$ admit uniformly positive densities on $[0,T]\times \T$ then there exists $C\in \R$ such that 
         \[
         \|\rho_i\|_{L^2([0, T];BV(\T))}< C
         \] for each $i \in \{1,2\}$.  
         In particular, the constant $C$ depends only on $\rho_{1,0},\rho_{2,0},V_1,V_2$ and $T$. 
    \end{corollary}
    \begin{proof}
        If $(\rho_1,\rho_2)$ is a smooth solution to the $\eta$-Regularised System with uniformly positive densities then it follows from Proposition \ref{changevar} that $(\sigma,f(r))$ solves System \eqref{eq:sigmaeq}, \eqref{eq:f(r)}. 
        Further, as the initial datum are assumed to be uniformly positive and smooth, it follows that 
        \[
        \partial_x \log\bigg(\frac{\rho_{1,0}}{\rho_{2,0}}\bigg) = \partial_x f(r_0) \in L^1(\T),
        \]
        and hence, from Theorem \ref{mainestimate}, it follows that there exists $A_1 \in \R$ depending only on $V_1,V_2,T$ and $TV(f(r_0))$ such that 
        \[
        \|\partial_x f(r)\|_{L^\infty([0,T];L^1(\T))}< A_1.
        \]
        Moreover, recalling that $f'(r) \geqslant 4$ for $r\in [0,1]$, the following inequality holds.
        \begin{equation}\label{eq:bound1}
        \begin{split}
        \|\partial_x r\|_{L^\infty([0,T];L^1(\T))} & \leqslant \inf_{\tau\in [0,1]}f'(\tau) \|\partial_x r\|_{L^\infty([0,T];L^1(\T))} 
        \\
        & \leqslant \|\partial_xf(r)\|_{L^\infty([0,T];L^1(\T))} \leqslant A_1.
          \end{split}
        \end{equation}
        Similarly, it is shown in Theorem \ref{sumestimate} that, when $\|\sigma_0\|_{L\log L(\T)}<+ \infty$, there exists a constant $A_2\in \R$, depending only on $V_1, V_2, T$ and $\|\sigma_0\|_{L\log L(\T)}$, such that
        \begin{equation}\label{eq:bound2}
        \|\sigma\|_{L^2([0,T];BV(\T))} \leqslant A_2.
        \end{equation}
        
        Since we are working in $1$-dimension, we have the continuous embedding $BV(\T)\hookrightarrow L^\infty(\T) $. 
        Consequently, there exists a constant $A_3>0$ (which is independent of $\sigma$) for which
        \begin{equation}\label{eq:bound3}
        \|\sigma\|_{L^2([0,T];L^\infty(\T))} \leqslant A_3\|\sigma\|_{L^2([0,T];BV(\T))} \leqslant A_2 A_3.
        \end{equation}
        
        Collecting the bounds obtained in Inequalities \eqref{eq:bound1}, \eqref{eq:bound2} and \eqref{eq:bound3}, we obtain the following estimate on the product $\partial_x\rho_1 = \partial_x (\sigma r)$.
        \begin{flalign*}
         \|\partial_x\rho_1\|_{L^2([0,T];L^1(\T))} & = \|\partial_x (\sigma r)\|_{L^2([0,T];L^1(\T))}
         \\
            & \leqslant \|\sigma\|_{L^2([0,T];L^\infty(\T))}\|\partial_x r\|_{L^\infty([0,T];L^1(\T))} \\
            & + \|r\|_{L^\infty([0,T]\times \T)}\|\partial_x \sigma\|_{L^2([0,T];L^1(\T))}
            \\
            & \leqslant (A_1A_3+1)A_2.
        \end{flalign*}
        
        The claim for $i = 2$ then follows from the application from the triangle inequality.
         \[
        \|\partial_x\rho_2\|_{L^2([0,T];L^1( \T))} \leqslant \|\partial_x\rho_1\|_{L^2([0,T];L^1( \T))}+ \|\partial_x\sigma\|_{L^2([0,T];L^1( \T))} \leqslant (A_1A_3+2)A_2.
        \]

    Lastly, since $\rho_i$ is a curve of probability measures, its $L^1$ norm satisfies 
    \[
    \|\rho_i\|_{L^2([0,T];L^1(\T))} = T^{\frac{1}{2}}.
    \]
    Moreover, the $L^2([0,T];BV(\T))$ norm is then bounded by

     \[
     \|\rho_i\|_{L^2([0,T];BV(\T))} = \|\rho_i\|_{L^2([0,T];L^1(\T))} + \|\partial_x\rho_i\|_{L^2([0,T];L^1(\T))} \leqslant T^{\frac{1}{2}} +(A_1A_3 + i)A_2.
     \]
    \end{proof}
    
    \begin{remark}\label{dependency}
    To clearly present the proof of the above corollary, we omitted the specific dependence of the constant $C$ on the initial conditions and potentials.
    However, by tracking the constants $\alpha$ and $\beta$, introduced in the proof of Theorem \ref{mainestimate} and by studying Inequality \eqref{eq:Gronwall}, it can be seen that the constant $A_1$ may be chosen to depend only on the initial conditions through the quantity 
    \[
TV\bigg(\log\bigg(\frac{\rho_{1,0}}{\rho_{2,0}}\bigg)\bigg)
    \]
    and on the potentials via the quantities
    \begin{equation}\label{eq:potentialquantities}
     \|\partial_x(V_1-V_2)\partial_xV_1\|_{W^{1,1}(\T)}, \  \|V_1-V_2\|_{W^{2,2}(\T)}, \ \|V_1-V_2\|_{W^{1,\infty}(\T)}, \  \|V_2-V_1\|_{W^{3,1}(\T)}. 
    \end{equation}
    Similarly, from Theorem \ref{sumestimate}, it follows that the constant $A_2$ may be chosen to only depend on the initial data through the quantity $\|\rho_{1,0}+\rho_{2,0}\|_{L\log L(\T)}$,
    and on the potentials through the quantities \begin{equation}\label{eq:potentialquantities2}
     \|V_1\|_{W^{1,\infty}(\T)}, \|V_2\|_{W^{1,\infty}(\T)}.
     \end{equation}
    Moreover, since we are working in $1$-dimension, the continuous embeddings 
    \[
    W^{3,1}(\T) \hookrightarrow W^{2,2}(\T) \hookrightarrow W^{1,\infty}(\T) \text{ and } W^{2,1}(\T) \hookrightarrow W^{1,\infty}(\T)
    \]
    follow from the Sobolev Embedding Theory and hence, each norm in \eqref{eq:potentialquantities} and \eqref{eq:potentialquantities2} may be controlled solely in terms of $\|V_1\|_{W^{2,1}(\T)}$, $\|V_2\|_{W^{2,1}(\T)}$ and $\|V_1-V_2\|_{W^{3,1}(\T)}$. 
    Subsequently, the constants $A_1, A_2$ may be chosen such that they only depend on the potentials through these latter three norms.
    \end{remark}
        
    To be clear, in establishing the main existence theorem, the purpose of the $BV$ estimate is to grant strong convergence of the individual species when we pass from the $\eta$-Regularised System to the limit as $\eta \to 0$. 
    This strong convergence will ultimately be granted as a consequence of the Aubin--Lions--Simon Compactness Theorem of  (see \cite[Section 8; Page 85; Corollary 4]{AubinLionsSimon}), for which we also require the following control on the time derivatives $\partial_t \rho_i$.
    \begin{proposition}\label{timecontrol}
         Let $(\rho_1,\rho_2)$ be a smooth solution to the $\eta$-Regularised System \eqref{eq:regcdid}. 
         There exists $s>0$ and $C \in \R$ such that 
         \[
         \|\partial_t \rho_i\|_{L^\frac{3}{2}([0,T];H^{-s}(\T))}\leqslant C.
         \] 
         Moreover, the constant $C$ may be taken to depend only on $\|\rho_{1,0}+\rho_{2,0}\|_{L\log L(\T)}$, $\|V_1\|_{W^{1,\infty}(\T)},$ $\|V_2\|_{W^{1,\infty}(\T)}$ and $T$.        
    \end{proposition}
    \begin{proof}
    Since $\rho_1,\rho_2$ are smooth and satisfy System \eqref{eq:regcdid}, the following equality is satisfied for any $\varphi \in C^\infty([0,T]\times \T)$.
    \begin{equation}\label{eq:Aubin0}
    \begin{split}
       \int_{0}^T\int_{\T}\partial_t\rho_i \varphi \diff x \diff t & = \int_{0}^T\int_{\T}\varphi\eta\partial_{xx}^2\rho_i + \varphi\partial_x(\rho_i\partial_x((1-2\eta)\log(\rho_1+\rho_2)+V_i)) \diff x \diff t   
        \\  
         &=\int_{0}^T\int_{\T}\eta\partial_{xx}^2\varphi\rho_i - \partial_x\varphi(\rho_i\partial_x((1-2\eta)\log(\rho_1+\rho_2)+V_i)) \diff x \diff t. \\
        \end{split}
    \end{equation}
    To bound the above equality, we first recognise that $\rho_i$ is a curve of probability measures and hence belongs to $L^\infty([0,T];L^1(\T))$ with norm equal to 1. 
    Secondly, we recognise that the logarithmic derivative satisfies the equality
    \[\rho_i\partial_x\log(\rho_1+\rho_2) = \frac{\rho_i}{\rho_1+\rho_2}\partial_x(\rho_1+\rho_2)\] 
    where, in particular, the quotient $\rho_i(\rho_1+\rho_2)^{-1}$ is bounded in the unit interval.
    As a consequence of these two statements, in conjunction with the fact that $\eta, (1-2\eta) \in [0,1]$, we deduce the following from H\"older's inequality.
    \begin{align*} \int_{0}^T& \int_{\T} \eta\partial_{xx}^2\varphi\rho_i - \partial_x\varphi(\rho_i\partial_x((1-2\eta)\log(\rho_1+\rho_2)+V_i)) \diff x \diff t
    \\
    & \leqslant \|\partial_{xx}^2\varphi\|_{L^1([0,T];L^\infty(\T))} +\|\partial_x\varphi\partial_x(\rho_1+\rho_2)\|_{L^1([0,T]\times\T)} + \|\partial_x\varphi\|_{L^1([0,T];L^{\infty}(\T))} \|V_i\|_{W^{1,\infty}(\T)}\\
    & \leqslant \|\partial_{xx}^2\varphi\|_{L^1([0,T];L^\infty(\T))} +\|\partial_x\varphi\|_{L^3([0,T]\times\T)}\|\partial_x(\rho_1+\rho_2)\|_{L^\frac{3}{2}([0,T]\times\T)} \\
    & + \|\partial_x\varphi\|_{L^1([0,T];L^{\infty}(\T))}\|V_i\|_{W^{1,\infty}(\T)}
    \end{align*}
    Since the domain $[0,T]\times \T$ is bounded, it further follows from Jensen's inequality that each of the norms 
    \[
    \|\partial_{xx}^2\varphi\|_{L^1([0,T];L^\infty(\T))}, \ \|\partial_x\varphi\|_{L^3([0,T]\times\T)}, \text{ and } \|\partial_x\varphi\|_{L^1([0,T];L^{\infty}(\T))}
    \]
    can be dominated by $\|\varphi\|_{L^3([0,T];W^{2,\infty}(\T))}$. Moreover, there exists a constant $\Lambda$ which depends only on $T$ such that 
    \begin{equation}\label{eq:Aubin1}
            \int_{0}^T\int_{\T}\partial_t\rho_i \varphi \diff x \diff t \leqslant \Lambda \|\varphi\|_{L^3([0,T];W^{2,\infty}(\T))}\bigg(1+\|\partial_x(\rho_1+\rho_2)\|_{L^\frac{3}{2}([0,T]\times\T)} + \|V_i\|_{W^{1,\infty}(\T)}\bigg)
    \end{equation}
    From Theorem \ref{sumestimate}, it follows that the quantity
    \[
    \|\partial_x(\rho_1+ \rho_2)\|_{L^\frac{3}{2}([0,T]\times\T)} 
    \]
    is bounded and, in particular, the bound on this quantity depends only on $\|\rho_{1,0}+\rho_{2,0}\|_{L\log L(\T)}$, $\|V_1\|_{W^{1,\infty}(\T)},\|V_2\|_{W^{1,\infty}(\T)}$ and $T$.
    Subsequently, Inequality \eqref{eq:Aubin1} implies that $\partial_t\rho_i$ is bounded in the space
    \[(L^3([0,T];H^s(\T)))^{*} \cong L^\frac{3}{2}([0,T];H^{-s}(\T)),
    \] for any choice of $s>0$ such that $H^s(\T) \hookrightarrow W^{2,\infty}(\T)$. Moreover, 
    this bound depends only on $\|\rho_{1,0}+\rho_{2,0}\|_{L\log L(\T)}$, $\|V_1\|_{W^{1,\infty}(\T)},\|V_2\|_{W^{1,\infty}(\T)}$ and $T$.
    \end{proof}

    We have established a $BV$ bound for the species solving the $\eta$-Regularised system and a bound for their time derivatives in a suitable negative Sobolev space. 
    We now prove our main theorem by applying the Aubin-Lions-Simon Lemma and passing to the limit as $\eta \to 0$. 
    \subsection{Proof of the Main Theorem}
    \begin{proof}[Proof of Theorem \ref{maintheorem}]
        For $i \in \{1,2\}$, let $(\rho_{i,0}^\eta)_{\eta \in (0,\frac{1}{2}]}$ denote a family of probability measures such that $\rho_{i,0}^\eta$ converges weakly to $\rho_{i,0}$ as $\eta \to 0$ and such that, for each $\eta \in (0,\frac{1}{2}]$, $\rho_{1,0}^\eta,\rho_{2,0}^\eta$ admit uniformly positive densities belonging to $C^\infty(\T)$.
        Since we assume that the Hypotheses \ref{H1} and \ref{H2} hold for the initial datum $\rho_{1,0},\rho_{2,0}$, we may further assume that the following quantities are bounded from above, uniformly in $\eta$.
    \begin{equation}\label{eq:unifenergy}
    \bigg\|\log\bigg(\frac{\rho_{1,0}^\eta}{\rho_{2,0}^\eta}\bigg)\bigg\|_{BV(\T)}, 
        \ \ \ \|\rho_{1,0}^\eta+\rho_{2,0}^\eta\|_{L\log L(\T)}.
    \end{equation}
    Additionally, for $i \in \{1,2\}$, let $(V_i^\eta)_{\eta \in (0,\frac{1}{2}]}$ denote a family of potentials such that $V_i^\eta$ converges weakly to $V_i$ as $\eta \to 0$ and such that $V_i^\eta \in C^\infty(\T)$. 
    Since we assume that the Hypotheses \ref{H3} and \ref{H4} hold for the potentials $V_1,V_2$, we may further assume that the following quantities are bounded, uniformly in $\eta$.
    \begin{equation}\label{eq:unifpotential}
        \|V_i^\eta\|_{W^{2,1}(\T)}, \ \ 
    \ \|V_1^\eta-V_2^\eta\|_{W^{3,1}(\T)}.
    \end{equation}
Since $\rho^\eta_{1,0},\rho^\eta_{2,0}$ admit smooth uniformly positive densities, it follows from Proposition \ref{smoothness} that the solutions $(\rho^\eta_1,\rho^\eta_2)$ are smooth and strictly uniformly positive on $[0,T]\times \T$.
    Moreover, as a consequence of Corollary \ref{maincorollary}, there exists, for each $\eta$, a constant $C_\eta>0$ such that 
    \[
    \|\rho_1^\eta\|_{L^2([0,T];BV(\T))} +  \|\rho_2^\eta\|_{L^2([0,T];BV(\T))} \leqslant C_\eta    \]
    and, as a consequence of Proposition \ref{timecontrol}, there exists a fixed $s \geqslant 0$ such that, for each $\eta$, there exists a constant $\hat{C}_\eta>0$ for which
     \[
    \|\partial_t {\rho_1^{\eta}}\|_{L^\frac{3}{2}([0,T];H^{-s}(\T))}+   \|\partial_t {\rho_2^{\eta}}\|_{L^\frac{3}{2}([0,T];H^{-s}(\T))}\leqslant \hat{C}_\eta.
    \] 
    As shown in Proposition \ref{timecontrol} and discussed in Remark \ref{dependency}, the constants $C_\eta, \hat{C}_\eta$ can be chosen to depend only $T$ along with the quantities introduced in \eqref{eq:unifenergy} and \eqref{eq:unifpotential}.
    Moreover, since we assumed that the quantities introduced in \eqref{eq:unifenergy} and \eqref{eq:unifpotential} are appropriately bounded, uniformly in $\eta$, the constants $C_\eta, \hat{C}_\eta$ may thus be chosen uniformly in $\eta$. 
    In particular, there exists $\tilde{C} > 0 $ such that, for any $\eta \in (0,\frac{1}{2}]$, the following inequality holds.
    \[
     \sum_{i=1}^2\bigg(\|\partial_t {\rho_i^{\eta}}\|_{L^\frac{3}{2}([0,T];H^{-s}(\T))}+ \|\rho_i^\eta\|_{L^2([0,T];BV({\T}))}\bigg)
    \leqslant \tilde{C}.       
    \]
    The compactness of the individual species now follows from an application of the Aubin--Lions--Simon Compactness Lemma. 
    In particular, the family $(\rho_i^\eta)_{\eta_\in (0,\frac{1}{2}]}$ is bounded in the space
    \[
    U^s := \{u \in L^2([0,T];BV(\T)) \ | \ \partial_t u \in L^\frac{3}{2}([0,T];H^{-s}(\T))\}.
    \]
    The space $BV(\T)$ is compactly embedded in $L^1(\T)$ whilst $L^1(\T)$ is continuously embedded in $H^{-s}(\T)$. 
    Hence, as a consequence of the classical Aubin--Lions--Simon Lemma \cite[Section 8; Page 85; Corollary 4]{AubinLionsSimon}, it follows from their boundedness in $U^s$ that the families $(\rho_i^\eta)_{\eta_\in (0,\frac{1}{2}]}$ are strongly pre-compact in $L^2([0,T];L^1(\T))$.

        \medskip

    As a consequence of this strong pre-compactness, there exists a subsequence that we do not relabel for which each $(\rho_i^\eta)_{\eta_\in (0,\frac{1}{2}]}$, $i=1,2$, which converges point-wise almost everywhere on $[0,T]\times \T$ to a limiting function $\rho_i$ as $\eta \to 0$, $i=1,2$.
    Moreover, along a suitable subsequence, the following quotients must similarly satisfy this point-wise almost everywhere convergence. 
    In particular, the following equality holds for almost every $(x,t)$ in $[0,T]\times \T$.
    \[
    \lim_{\eta \to 0}\rho_i^\eta(x,t)(\rho_1^\eta(x,t)+\rho_2^\eta(x,t))^{-1} = \rho_i(x,t)(\rho_1(x,t)+\rho_2(x,t))^{-1}.
    \]
    On the other hand, the quotients $\rho_i^\eta(\rho_1^\eta+\rho_2^\eta)^{-1}$ 
    are bounded uniformly in the unit interval and so, by the Uniform Boundedness Theorem, the convergence of these quotients must also hold strongly in $L^p([0,T]\times{\T})$ for any $p\in [1,\infty)$.

    \medskip
    
    Having addressed the convergence of the individual species we now address the convergence of the gradient of the sum. 
    In particular, since $\|\rho_{1,0}^\eta+\rho_{2,0}^\eta\|_{L\log L(\T)}$ is assumed to be bounded from above, uniformly in $\eta$, it follows from Theorem \ref{sumestimate} that the family $(\partial_x(\rho_1^\eta+\rho_2^\eta))_{\eta\in (0,\frac{1}{2}]}$ is bounded in $L^\frac{3}{2}([0,T]\times \T)$, uniformly in $\eta$, and hence, this family is also pre-compact for the weak convergence in $L^\frac{3}{2}([0,T]\times \T)$ as a consequence of the Banach--Alaoglu Theorem (this space is reflexive so the weak-$^*$ and weak convergence coincide).
    Moreover, in considering the sub-sequence along which the point-wise convergence of the individual species occurs, we may further extract a subsequence for which $\partial_x(\rho_1^\eta+\rho_2^\eta)$ converges weakly in $L^\frac{3}{2}([0,T]\times \T)$ to $\partial_x(\rho_1+\rho_2)$ as $\eta \to 0$.

    \medskip
    
    As the product of a strongly convergent sequence and a weakly convergent sequence, it subsequently follows that the expression 
    \begin{equation}\label{eq:momentproduct}
        \rho_i^\eta(\rho_1^\eta+\rho_2^\eta)^{-1}\partial_x(\rho_1^\eta+\rho_2^\eta)
    \end{equation}
    must converge weakly in $L^q([0,T]\times\T)$ to 
    \[
    \rho_i(\rho_1+\rho_2)^{-1}\partial_x(\rho_1+\rho_2) = \rho_i\partial_x\log(\rho_1+\rho_2)
    \] 
    as $\eta \to 0$, for any $q \in [1,\frac{3}{2})$.

    \medskip
    
    Before taking the limit as $\eta \to 0$, we finally address the compactness of the potential term. 
    In particular, since we assumed that the family $V_i^\eta$ is uniformly bounded in $W^{2,1}(\T)$ and since we assumed that $V_i^\eta$ converges weakly to $V_i$ as $\eta \to 0$, it follows from the Rellich--Kondrachov compactness theorem that $V_i^\eta$ converges strongly in $W^{1,p}(\T)$ to $V_i$ as $\eta \to 0$ for any $p \in [1,\infty)$. 

    \medskip
    
    Ultimately, in taking the limit as $\eta \to 0$ on either side of Equation \eqref{eq:weaketacdid} and, using the established convergence of the cross-diffusive terms \eqref{eq:momentproduct} along with the convergence of the potentials, we conclude that the following equality holds for $i \in \{1,2\}$ and any $\varphi \in C_{c}^\infty([0,T)\times \T)$.
    \begin{flalign*}
    \lim_{\eta\to 0}\bigg(\int_{\T} \rho_{i,0}^\eta\varphi_0 \diff x & + \int_0^T\int_{\T}\rho_i^\eta\partial_t\varphi \diff x \diff t\bigg) \\
       = \int_{\T} \rho_{i,0}\varphi_0 \diff x & + \int_0^T\int_{\T}\rho_i\partial_t\varphi \diff x \diff t = \int_0^T\int_{\T} \rho_i\partial_x(\log(\rho_1+\rho_2)+V_i)\partial_x \varphi \diff x \diff t
        \\
        & = \lim_{\eta\to 0}
        \int_0^T\int_{\T} \rho_i^\eta\partial_x((1-2\eta)\log(\rho_1^\eta+\rho_2^\eta)+V_i^\eta)\partial_x \varphi - \eta\rho_i^\eta \partial_{xx}^2\varphi \diff x \diff t.
    \end{flalign*}
    The above equality establishes that $\rho_1,\rho_2$ are weak solutions of System \eqref{eq:cdid} with the exception of the inclusion $\rho_1,\rho_2 \in C([0,T];\sP(\T)) \cap L^p_{loc}((0,T];BV(\T))$ for every $p \in [1,\infty)$. 
    We now establish this inclusion, first addressing the continuity in $\sP(\T)$.

    \medskip
    
    Since we are in the 1-dimensional setting, the continuous embedding $\mathscr{M}(\T) \hookrightarrow H^{-s}(\T)$ holds for any $s\geqslant 1$. 
    On the other hand, $\T$ is compact and hence, $\sP(\T)$ is also compact. 
    Moreover, the following continuous embedding holds for any $s \geqslant 1, \ p > 1$ as a consequence of the Aubin--Lions--Simon Lemma \cite[Section 8; Page 85; Corollary 4]{AubinLionsSimon}.
    \[
    \hat{U}^s := \{u \in L^\infty([0,T];\sP(\T)) \ | \ \partial_t u \in L^p([0,T];H^{-s}(\T))\} \hookrightarrow C([0,T];\sP(\T)).
    \]
    In particular, $\rho_i \in L^\infty([0,T];\sP(\T))$ and, as previously remarked, there exists $s \geqslant 1$ such that $\partial_t\rho_i \in L^\frac{3}{2}([0,T];H^{-s}(\T))$ as a consequence Proposition \ref{timecontrol}. 
    This means that $\rho_i \in \hat{U}^s$ for an appropriate choice of $s$ and so, we conclude that $\rho_i \in C([0,T];\sP(\T))$.
    Moreover, it is left to show that $\rho_1,\rho_2 \in L^p_{loc}((0,T];BV(\T))$ for every $p \in [1,\infty)$.  

    \medskip

    The limiting sum $\sigma = \rho_1 + \rho_2$ weakly solves the linear Fokker--Planck Equation 
    \[
    \begin{cases}
    \partial_t\sigma = \partial_x(\partial_x\sigma + b\sigma), \\ 
    b = \partial_xV_1 \frac{\rho_1}{\rho_1+\rho_2}+ \partial_xV_2 \frac{\rho_2}{\rho_1+\rho_2}, 
    \end{cases}
    \text{ on } (0,T)\times \T.
    \]
    Since $b \in L^\infty([0,T]\times\T)$, it follows from \cite[Page 256, Corollary 6.4.3.]{bogachev2015fokker} that $\sigma \in L^p_{loc}((0,T];W^{1,p}(\T))$ for every $ p \in [1,\infty)$ and, since the domain $\T$ is compact, we also deduce that $\sigma \in L^p_{loc}((0,T];BV(\T))$ for every $p \in [1,\infty)$.
    From the uniform estimates of Theorem \ref{mainestimate}, we know that 
    \[
    \frac{\rho_1}{\sigma} = r \in L^\infty([0,T];BV(\T)).
    \]
    Subsequently, it follows that $\rho_1,\rho_2\in L^p_{loc}((0,T];BV(\T))$ for every $p \in [1,\infty)$ as the product of $r$ (respectively $(1-r)$) and $\sigma$.
    \end{proof}

        In the scenario $\rho_{1,0},\rho_{2,0} \in L^q(\T)$, the final regularity of argument of the preceding proof applies globally. 
        In particular, if $\rho_{1,0}+\rho_{2,0} \in L^q(\T)$ then a standard energy dissipation argument for the energy $\sigma \mapsto \|\sigma\|_{L^q(\T)}^q$ shows that $\sigma \in L^\infty([0,T];L^q(\T))$ (see, for instance, Equation \eqref{eq:reacede} for an exposition of this argument in the case $q=2$).
        The sum $\sigma$ may thus be recast as a solution to the forced heat equation 
        \begin{equation}\label{eq:remarkheatsum}
      \begin{cases}
         \partial_t\sigma =\partial_{xx}^2\sigma + \partial_x\xi, \\ 
         \xi = \partial_xV_1\rho_1+ \partial_xV_2\rho_2, 
        \end{cases}
        \text{ on } (0,T)\times \T.  
    \end{equation}
        Since $\sigma \in L^\infty([0,T];L^q(\T))$ it follows that $\xi \in L^\infty([0,T];L^q(\T))$. 
        Moreover, by applying Proposition \ref{forcedheat}, it follows that 
        $\sigma \in L^q([0,T];W^{1,q}(\T))\subset L^q([0,T];BV(\T))$. 
        Since we also know that $r \in L^\infty([0,T];BV(\T))$ it then follows that $\rho_1,\rho_2 \in L^q([0,T];BV(\T))$ thus, extending the regularity of solutions in the case $q > 2$.
\section{On the Smoothness of the Regularised System}\label{sec:four}
The work of Laborde \cite{LabordeCDS} allowed us to establish a weak existence theory for System \eqref{eq:regcdid}. 
We now establish that such weak solutions are smooth and uniformly positive under suitable regularity assumptions on the initial data and the potentials. 
These results are established by means of a standard parabolic estimate (Proposition \ref{forcedheat}) and the subsequent estimate on the positivity of solutions which follows from Harnack's inequality (Proposition \ref{lowerbound}). In the lack of a precise classical reference that one would be able to use in our setting, we decided to provide the full details on these results here below.

\begin{proposition}\label{forcedheat}
    Let $k \in \N\cup \{0\}, \ p \in (1,\infty)$ and $\eta > 0$. 
    Let $\xi \in L^p([0,T];W^{k,p}(\T))$, let $u_0\in \sP(\T)$ and let $u \in L^p([0,T]\times \T)$ denote a very weak solution to the forced heat equation
    \begin{equation}\label{eq:forcedheat}
    \partial_tu = \eta\partial_{xx}^2 u + \partial_x\xi \text{ on } (0,T)\times\T 
    \end{equation}
    in the sense that the following equality holds for every $\varphi \in C_c^\infty([0,T)\times \T)$.
    \begin{equation}\label{eq:weakforcedheat}
    \int_{\T} \varphi_0 u_0 \diff x + \int_0^T\int_{\T}(\partial_t\varphi + \partial_{xx}^2\varphi) u \diff x \diff t = \int_0^T\int_{\T} \partial_x \varphi \xi \diff x \diff t.
    \end{equation}
    If $u_0 \in W^{k,p}(\T)$ then $u \in L^p([0,T];W^{k+1,p}(\T))$. 
\end{proposition}
\begin{proof}
We prove the claim in the case $k=0, \eta = 1$. Moreover, the case $k \in \N$ may be treated inductively by differentiating Equation \eqref{eq:forcedheat} in the spatial variable whilst the case $\eta \neq 1$ may be treated by a rescaling argument.
In addition, since $\T$ is compact, it is equivalent to establish the desired regularity in terms of local Sobolev estimates for the periodic representation of $u$ defined on the whole space, that is, an $L^p([0,T];W^{1,p}_{loc}(\R))$ estimate for the function $\tilde{u} \co [0,T] \times \R \to \R$ given for almost every $(t,x) \in  [0,T] \times \T $ by the equality
\[
    \tilde{u}(t,x+z) = u(t,x) \text{ for every } z\in \Z.
\]
It may readily be shown that $\tilde{u}$ is a very weak solution to the forced heat equation  \begin{equation}\label{eq:forcedheat2}
    \partial_t\tilde{u} = \partial_{xx}^2 \tilde{u} + \partial_x\tilde{\xi} \text{ on } (0,T)\times\R,
    \end{equation}
    with the forcing term $\tilde{\xi}$ given by 
    \[
    \tilde{\xi}(t,x+z) = \xi(t,x) \text{ for every } z\in \Z.
    \]
    Now that the equation is set on the whole space, we derive a Duhamel formulation for the solution which, in conjunction with the heat kernel estimates established in \cite[Chapter IV]{Ladyzhenskaya}, allows us to improve the regularity of the solutions via a parabolic bootstrapping argument.

    \medskip
    
    Denoting the heat kernel on $\R$ by $\Phi\co[0,T]\times \R\to \R$, it is shown in \cite[Page 17-18]{Jordan1998} that a weak solution $\tilde{u}$ satisfies the following Duhamel formula for every smooth cut-off function $\varphi \in C_c^\infty(\R)$, almost everywhere on $[0,T]\times \R$. 
    This Duhamel formula is derived by taking the backwards in time heat kernel multiplied by the smooth cut-off function $\varphi$ as a test function in the very weak formulation \eqref{eq:weakforcedheat}.
    In particular, whilst \cite{Jordan1998} treats the linear Fokker--Planck equation, the argument for the forced heat equation is completely analogous. 

    \begin{equation}\label{eq:duhamel}
    (\tilde{u}\varphi)  = (\tilde{u}\partial_{xx}^2\varphi - \tilde\xi \partial_x \varphi) \ast_2 \Phi +\partial_x[(2\tilde{u}\partial_{x}\varphi - \tilde\xi \varphi) \ast_2 \Phi]
    + (\tilde{u}_0\varphi) \ast \Phi.
    \end{equation}

    Here, the operation $\ast$ denotes the standard convolution in the spatial variable whilst the operation $\ast_2$ is given by
    \[
    (u_1 \ast_2 u_2)(t,x) := \int_0^t \int_\R u_1(s,y)u_2(t-s,x-y) \diff y \diff s.
    \]
    
    Since the heat kernel $\Phi$ is smooth on $(0,T)\times \R$, any function convoluted with $\Phi$ is also smooth on $(0,T)\times \R$. 
    Moreover, we are justified in differentiating Equation \eqref{eq:duhamel}, which thus yields the expression
    \[
    \partial_x(\tilde{u}\varphi)  = \partial_x[(\tilde{u}\partial_{xx}^2\varphi - \tilde\xi \partial_x \varphi) \ast_2 \Phi] +\partial_{xx}^2[(2\tilde{u}\partial_{x}\varphi - \tilde\xi \varphi) \ast_2 \Phi]
    + (\tilde{u}_0\varphi)\ast \partial_x\Phi.    
    \]
    Subsequently, from the triangle inequality, we deduce the inequality
    \begin{equation}\label{eq:diffduhamel}
    \begin{split}
        \|\partial_x(\tilde{u}\varphi)\|_{L^p([0,T]\times \R)}  & \leqslant \|(\tilde{u}_0\varphi)\ast \partial_x\Phi\|_{L^p([0,T]\times \R)} \\
        & + \|\partial_x[(\tilde{u}\partial_{xx}^2\varphi - \tilde\xi \partial_x \varphi) \ast_2 \Phi] \|_{L^p([0,T]\times \R)} \\
        & + \|\partial_{xx}^2[(2\tilde{u}\partial_{x}\varphi - \tilde\xi \varphi) \ast_2 \Phi]\|_{L^p([0,T]\times \R)}.
    \end{split}
    \end{equation}
     Now, using the estimate provided in \cite[Page 288; Chapter IV; $\mathsection 3$, Equation 3.1]{Ladyzhenskaya} in conjunction with Inequality \eqref{eq:diffduhamel}, we conclude the existence of a constant $c_p$ for which the following inequality holds.
     \begin{equation}\label{eq:diffduhamel2}
     \begin{split}
         \|\partial_x(\tilde{u}\varphi)\|_{L^p([0,T]\times \R)}  \leqslant & \|(\tilde{u}_0\varphi)\ast \partial_x\Phi\|_{L^p([0,T]\times \R)} 
        \\
        + & c_p\|(\tilde{u}\partial_{xx}^2\varphi - \tilde\xi \partial_x \varphi)\|_{L^p([0,T]\times \R)} + c_p\|(2\tilde{u}\partial_{x}\varphi - \tilde\xi \varphi) \|_{L^p([0,T]\times \R)}.
     \end{split}
    \end{equation}
Since it is assumed that $u,\xi \in L^p([0,T]\times \T)$, it follows that $\tilde{u}, \tilde{\xi} \in L^p([0,T];L^p_{loc}(\R))$ and hence, the second line of Equation \eqref{eq:diffduhamel2} is finite and it is left to address the term in the right hand side of the first line. 
In particular, using Young's convolution inequality in conjunction with the explicit formula for the heat kernel, we conclude the existence of a constant $C$ for which the following inequality holds.
 \begin{flalign*}
     \|(\tilde{u}_0\varphi)\ast \partial_x\Phi\|_{L^p([0,T]\times \R)}& \leqslant 
     \|(\tilde{u}_0\varphi)\|_{L^p([0,T]\times \R)} \|\partial_x\Phi\|_{L^1([0,T]\times\T)} \\
     & \leqslant C\|(\tilde{u}_0\varphi)\|_{L^p([0,T]\times \R)} \bigg|\int_0^T 2t^{-\frac{1}{2}}\bigg|
     \\
     & \leqslant C T^\frac{1}{2}\|(\tilde{u}_0\varphi)\|_{L^p([0,T]\times \R)}.
 \end{flalign*}
Combining the above inequality with Equation \eqref{eq:diffduhamel2}, it follows that
\[
\|\partial_x(\tilde{u}\varphi)\|_{L^p([0,T]\times\R)} < +\infty.
\]
Further, since $\varphi$ was taken to be an arbitrary smooth cut-off function, we conclude that 
\[\tilde{u} \in L^p([0,T];W^{1,p}_{loc}(\R)).\] 
Now, since $\tilde{u} \in L^p([0,T];W^{1,p}_{loc}(\R))$, we conclude the desired regularity for $u$ from the following inequality, which holds for any set $K$ containing a unit interval.
\begin{flalign*}
\|u\|_{L^p([0,T]; W^{1,p}(\T))} & \leqslant \|\tilde{u}\|_{L^p([0,T];W^{1,p}(K))} < +\infty.
\end{flalign*}
\end{proof}

\begin{proposition}\label{lowerbound}
    Let $\eta >0$, let $b \in L^\infty([0,T]\times \T)$, let $u_0 \in C^{0,\alpha}(\T); \alpha \in (0,1]$ and let $u$ denote the weak solution of the Fokker--Planck Equation 
    \begin{equation}\label{eq:FP}
        \partial_t u = \partial_x(\eta\partial_x u + bu) \text{ on } (0,T) \times \T
    \end{equation}
    in the sense that $u \in L^2([0,T];H^1(\T))\cap L^\infty([0,T];L^2(\T))$ and the following equality holds for every $\varphi \in C_c^\infty([0,T)\times \T)$.
    \begin{equation}\label{eq:weakfp}
     \int_{\T} \varphi_0 u_0 \diff x + \int_0^T\int_{\T}\partial_t\varphi u \diff x \diff t =  \int_0^T\int_{\T} \partial_{x} \varphi(\eta  \partial_xu +  b u) \diff x \diff t.
    \end{equation}
    There exists $\beta \in (0,1]$ such that $u \in C^{0,\beta}([0,T]\times\T)$. 
    Moreover, if $u_0$ is strictly uniformly positive on $\T$, then $u$ is strictly uniformly positive on $[0,T]\times \T$.
\end{proposition}

\begin{proof}
Since $b \in L^\infty([0,T]\times \T)$ and since $u_0 \in C^{0,\alpha}(\T)$, it follows from \cite[Page 204; Theorem 10.1]{Ladyzhenskaya} that there exists $\beta \in (0,1]$ such that $u \in C^{0,\beta}([0,T]\times\T)$.
Moreover, if $u_0$ is assumed to be uniformly positive on $\T$ then, by the uniform H\"older continuity of $u$, there exists $\tau > 0$ and $c(\tau) > 0$ such that $u > c(\tau)$ on $[0,\tau]\times \T$. 

\medskip

It is left to establish the uniform positivity of $u$ on $[\tau,T]\times \T$. For this result, we adapt \cite[Theorem 8.1.3]{bogachev2015fokker} to the periodic setting.
In particular, the original statement of this theorem produces a Harnack-type inequality which holds locally in space, i.e. for compact subsets of an open and bounded domain on which $u$ defines a local solution. 
However, we recognise that, due to the spatial periodicity of our setting, we may take this result to hold globally in space.
Moreover, by tracking the constants introduced in the statement of \cite[Theorem 8.1.3]{bogachev2015fokker}, it follows that there exists a constant $K > 0$, which depends only on $\|b\|_{L^\infty([0,T]\times\T)}$ and $\eta$, such that, for every $0 < s < t \leqslant T$ and $x,y \in \T$, $u$ satisfies the inequality 
\begin{equation}\label{eq:harnackone}
u(s,y) \exp{\bigg(-K\bigg(\frac{d^2(x,y)}{t-s} + \frac{t-s}{\min\{s,1\}}+1\bigg) \bigg)}\leqslant u(t,x) .  
\end{equation}
Here $d(x,y) := \inf_{z\in \mathbb{Z}} |x-y +z|$ denotes the canonical metric on the flat-torus.

\medskip

Since we have already established a suitable uniform lower bound on the interval $[0,\tau]\times \T$, we are only concerned with a lower bound on the remaining set $[\tau,T]\times \T$.
Moreover, we infimise the left-hand side of Equation \eqref{eq:harnackone}, taking the infimum over the set 
\[
A^\tau:=\bigg\{(s,y),(t,x)  \bigg| (s,y) \in \left[\frac{\tau}{4},\frac{\tau}{2}\right]\times \T, \ (t,x) \in [\tau,T]\times \T\bigg\}.
\]
Since we know that $u > c$ on $\left[\frac{\tau}{4},\frac{\tau}{2}\right]\times \T$, the subsequent inequality follows.
\begin{equation}\label{eq:harnacktwo}
c \exp\bigg(-K\bigg(\frac{2}{\tau} + \frac{T}{\min\{\frac{\tau}{4},1\}}+1\bigg)\bigg) \leqslant \inf_{A^\tau}\left\{ u(s,y)\exp{\bigg(-K\bigg(\frac{d^2(x,y)}{t-s} + \frac{t-s}{\min\{s,1\}}+1\bigg) \bigg)}\right\}.  
\end{equation}
By plugging Inequality \eqref{eq:harnacktwo} into Equation \eqref{eq:harnackone}, we obtain the following inequality.

\begin{equation}\label{eq:harnackthree}
    0 < c \exp\bigg( -K\bigg(\frac{2}{\tau} + \frac{T}{\min\{\frac{\tau}{4},1\}}+1\bigg)\bigg) \leqslant \inf_{(t,x) \in A^\tau} u(t,x)  = \inf_{(t,x) \in [\tau,T]\times \T} u(t,x).
\end{equation}
Moreover, since $u > c $ on $[0,\tau] \times \T$ and since Inequality \eqref{eq:harnackthree} provides a positive lower bound on $[\tau,T] \times \T$, we conclude that $u$ is strictly uniformly positive on all of $[0,T]\times \T$.
\end{proof}

Utilising Propositions \ref{forcedheat} and \ref{lowerbound}, we now derive the desired regularity for System 
\eqref{eq:regcdid}.

\begin{proposition}\label{smoothness}
    Let $\eta \in \left(0,\frac{1}{2}\right]$ and let $\rho_1,\rho_2$ denote a weak solution of the $\eta$-Regularised System \eqref{eq:regcdid} with initial data  $\rho_{1,0},\rho_{2,0} \in \sP(\T)$ and potentials $V_1,V_2 \in C^\infty(\T)$.
    If $\rho_{1,0},\rho_{2,0}$ admit uniformly positive densities belonging to $C^\infty(\T)$ then $\rho_1,\rho_2 \in C^\infty([0,T]\times\T)$.
\end{proposition}
\begin{proof}
    Given $\rho_1,\rho_2$, a weak solution to the $\eta$-Regularised System \eqref{eq:regcdid}, the sum $\rho_1+\rho_2$, which we will from here on denote by $\sigma$, defines a weak solution to the Fokker-Planck equation 
    \[
    \begin{cases}
    \partial_t\sigma = \partial_x((1-\eta)\partial_x\sigma + b\sigma), \\ 
    b = \partial_xV_1 \frac{\rho_1}{\rho_1+\rho_2}+ \partial_xV_2 \frac{\rho_2}{\rho_1+\rho_2}, 
    \end{cases}
    \text{ on } (0,T)\times \T.
    \]
    Since $b\in L^\infty([0,T]\times \T)$ and since the initial datum $\rho_{1,0},\rho_{2,0}$ admit smooth uniformly positive densities, it follows from Proposition \ref{lowerbound} that $\sigma$ is uniformly H\"older continuous and uniformly positive on $[0,T]\times \T$.
    Moreover, since $\sigma \in C([0,T];L^p(\T))$and since $0 < \rho_i < \sigma$ it follows that $\rho_i \in L^\infty([0,T]\times \T)$.
    Using this summability on $\rho_i$, we initiate an iterative bootstrapping process for the regularity of $\rho_i$.
    In particular, since this argument hinges upon Proposition \ref{forcedheat}, we recast $\sigma$ as 
    a weak solution to the forced heat equation \begin{equation}\label{eq:heatsum}
    \begin{cases}
    \partial_t\sigma = (1-\eta)\partial_{xx}^2\sigma + \partial_x\xi, \\ 
    \xi = \partial_xV_1\rho_1+ \partial_xV_2\rho_2, 
    \end{cases}
    \text{ on } (0,T)\times \T.  
    \end{equation}
    Similarly, we recognise that $\rho_i$ weakly solves the forced heat equation 
     \begin{equation}\label{eq:heatspecies}
    \begin{cases}
    \partial_t\rho_i = \eta\partial_{xx}^2\rho_i + \partial_x\xi_i, \\ 
    \xi_i = \rho_i(\partial_x\log(\sigma)+\partial_xV_i), 
    \end{cases}
    \text{ on } (0,T)\times \T.   
     \end{equation}
    The bootstrapping argument can be expressed in the following steps, and we commence in the case $k=0$.
    \begin{enumerate}
        \item Since $\rho_i \in L^p([0,T];W^{k,p}(\T))$ it follows that $\xi \in L^p([0,T];W^{k,p}(\T))$, for every $p \in [1,\infty)$.
        \item Applying Proposition \ref{forcedheat} to Equation \eqref{eq:heatsum}, it follows that $\sigma\in L^p([0,T];W^{k+1,p}(\T))$ for every $p \in [1,\infty)$.
        \item Since $\sigma$ is uniformly positive on $[0,T]\times \T$, it follows that $\partial_x\log(\sigma)\in L^p([0,T];W^{k,p}(\T))$ and hence, $\xi_i \in L^p([0,T];W^{k,p}(\T))$ for every $p \in [1,\infty)$.
        \item Applying Proposition \ref{forcedheat} to Equation \eqref{eq:heatspecies}, it follows that $\rho_i \in L^p([0,T];W^{k+1,p}(\T))$ for every $p \in [1,\infty)$.
    \end{enumerate}
    By iterating the argument above, it is possible to establish that $\rho_i\in L^p([0,T];W^{k,p}(\T))$ for an arbitrary $k \in \N$ and for every $p \in [1,\infty)$.

    \medskip
    
    Once spatial regularity has been established up to an arbitrary order, regularity on the temporal derivative subsequently follows from Equation \eqref{eq:heatspecies}.
    For example, if we know that $\rho_i \in L^p([0,T];W^{2,p}(\T))$ then the right hand side of Equation \eqref{eq:heatspecies} belongs to $L^p([0,T]\times\T)$ and hence, $\partial_t\rho_i\in L^p([0,T]\times\T)$ and $\rho_i \in W^{1,p}([0,T]\times \T)$ for every $p \in [1,\infty)$. 
    Moreover, since we can show that $\rho_i \in W^{k,p}([0,T]\times \T)$ for arbitrary $k \in \N$ (and any $p \in [1,\infty)$), we conclude that $\rho_1, \rho_2 \in C^\infty([0,T]\times\T)$.

    \medskip 
    
    Finally, to conclude the uniform positivity of the densities $\rho_i$, we recognise that 
    defines a weak solution to the Fokker--Planck equation 
    \[
    \begin{cases}
    \partial_t\rho_i = \partial_x(\eta\partial_x\rho_i + b_i\rho_i), \\ 
    b_i = \partial_xV_i + \partial_x\log(\sigma), 
    \end{cases}
    \text{ on } (0,T)\times \T.
    \]
    In particular, since it is now clear that $b_i \in W^{1,2}([0,T]\times \T) \hookrightarrow L^\infty([0,T]\times\T)$ and since $\rho_{i,0}$ is assumed to uniformly strictly positive, it follows from Proposition \ref{lowerbound} that $\rho_i$ is uniformly strictly positive on $[0,T]\times \T$.
\end{proof}

\section{On the Inclusion of Reaction Terms}\label{sec:five}

Whilst it is not the main focus of this work, we believe that the methods outlined in this manuscript may readily be adapted in order to establish the weak existence theory for reaction-cross-diffusion systems of the type

\begin{equation}\label{eq:rcdid}
\begin{cases}
    \partial_t \rho_1 = \partial_x(\rho_1\partial_x(\log(\rho_1+\rho_2)+V_1)) + \rho_1 F_1(\rho_1,\rho_2),\\
    \partial_t \rho_2 = \partial_x(\rho_2\partial_x(\log(\rho_1+\rho_2)+V_2)) + \rho_2 F_2(\rho_1,\rho_2),
\end{cases}
\text{on } (0,T)\times \T.
\end{equation}

That being said, it is beyond the scope of this manuscript to construct an approximation of System \eqref{eq:rcdid}, as we did for System \eqref{eq:cdid} by constructing System \eqref{eq:regcdid}.
Instead, we present the formal computations necessary to derive the essential $BV$ estimates, directly on System \eqref{eq:rcdid}, and say that a smooth approximation is `suitable' for System \eqref{eq:rcdid} if these formal computations can be rigorously justified whilst preserving the integrity of the a priori $BV$ estimates, thus allowing for the passage to the limiting system. 
Our main result in this direction is as follows.

\begin{thm}\label{reactionexistence}
    Assume that there exists a suitable smooth approximation of System \eqref{eq:rcdid} which admits uniformly positive densities.
    Further, assume that $\rho_{1,0}, \rho_{2,0} \in \sP(\T) \cap L^\infty(\T)$.

    \medskip
    
    If the reaction terms $F_1,F_2\co [0,+\infty)^2 \to \R$ are bounded and Lipschitz continuous and Hypotheses \ref{H2}-\ref{H4} hold, then there exists a weak solution to System \eqref{eq:rcdid} with potentials $V_1,V_2$, reaction terms $F_1,F_2$ and initial data $\rho_{1,0},\rho_{2,0}$.
\end{thm}
\begin{proof}[Sketch of Proof]
As in Section \hyperref[sec:three]{Three}, we would like to make a change of variables $(\rho_1,\rho_2) \mapsto (\sigma, f(r))$ and estimate the quantities $\sigma, f(r)$ in some appropriate $L^p([0,T];BV(\T))$ space.
To make clear this change of variables we first define new reaction functions
\begin{equation}\label{eq:newreaction}
\tilde{F_i}(r,\sigma):= F_i(r\sigma,(1-r)\sigma).
\end{equation}
Subsequently, and following the calculations of Proposition \ref{changevar}, if $(\rho_1,\rho_2)$ solves System \eqref{eq:rcdid} then, at least formally, with the same choice of potentials $V$ and $W$ as in \eqref{def:VW},
the variable $\sigma$ satisfies the equality
\begin{equation}\label{eq:reacsigma}
    \partial_t\sigma = \partial_{x}(\partial_x\sigma+ \sigma W + r\sigma V) + r\sigma\tilde{F}_1(r,\sigma)
    + (1-r)\sigma\tilde{F}_2(r,\sigma).
\end{equation}
Meanwhile, the variable $f(r)$ satisfies the equality
\begin{equation}\label{eq:reacf}
\begin{split}
\partial_t f(r) & =(1-r)\partial_x f(r)V + (\partial_x f(r) + V)(\partial_x\log(\sigma) +W)\\
& -VW + \partial_x V +\tilde{F}_1(r,\sigma)-\tilde{F}_2(r,\sigma).
\end{split}
\end{equation}

Using Equation \eqref{eq:reacsigma}, we first establish Sobolev estimates for $\sigma$ via a classical energy dissipation argument (for which we refer the interested reader to \cite[Chapter III, $\mathsection 2$]{Ladyzhenskaya} for a more detailed exposition of this technique). 
In particular, we consider the dissipation of the $L^2$ energy $\sigma \mapsto \|\sigma\|_{L^2(\T)}^2$. 
This dissipation reads as follows.

 \begin{equation}\label{eq:reacede}
    \int_{\T} |\sigma_s|^2 \diff x = \int_{\T} |\sigma_0|^2  \diff x  + 2\int_0^s\int_{\T} \sigma_t \partial_t\sigma_t \diff x \diff t.
    \end{equation}

In order to establish the gradient bounds, we would like to once more apply a Gr\"onwall argument and hence, it remains to bound the righter most term in Equation \eqref{eq:reacede} in terms of the $L^2$-energy. 
In particular, substituting with Equation \eqref{eq:reacsigma} and integrating by parts, we obtain the following equality. 

 \begin{equation}\label{eq:rhs}
 \begin{split}
   \int_0^s\int_{\T} \sigma \partial_t\sigma \diff x \diff t & =  -\int_0^s\int_{\T} |\partial_x\sigma|^2 + \partial_x\sigma (\sigma W + r\sigma V) \diff x \diff t\\
   & + \int_0^s\int_{\T}  r\sigma^2\tilde{F}_1(r,\sigma)
    + (1-r)\sigma^2\tilde{F}_2(r,\sigma) \diff x \diff t.
 \end{split}
    \end{equation}

Applying Young's inequality to the right hand side in the first line of Equation \eqref{eq:rhs}, it further follows that 

 \begin{equation}\label{eq:rhsII}
 \begin{split}
   \int_0^s\int_{\T} \sigma \partial_t\sigma \diff x \diff t & \leqslant -\frac{1}{2}\int_0^s\int_{\T} |\partial_x\sigma|^2  \diff x \diff t + \frac{1}{2}\int_0^s\int_{\T}\sigma^2 ( W + r V)^2 \diff x \diff t\\
   & + \int_0^s\int_{\T}  r\sigma^2\tilde{F}_1(r,\sigma)
    + (1-r)\sigma^2\tilde{F}_2(r,\sigma)\diff x \diff t.
 \end{split}
    \end{equation}

From Hypothesis \ref{H3}, it is assumed that $W,V \in W^{1,1}(\T) \hookrightarrow L^\infty(\T)$, meanwhile, we also assumed that $F_i$ is uniformly bounded on $[0,+\infty)^2$. 
Moreover, since we assume that these calculations are justified for an approximation for which $\rho_1,\rho_2$ are non-negative, it follows that $r \in [0,1]$ and hence, by applying H\"older's inequality, it follows that there exists $C$ which depends only on $W,V,F_i$ such that 

 \begin{equation}\label{eq:rhsIII}
 \begin{split}
   \int_0^s\int_{\T} \sigma \partial_t\sigma \diff x \diff t \leqslant -\frac{1}{2}\int_0^s\int_{\T} |\partial_x\sigma|^2  \diff x \diff t + \frac{C}{2}\int_0^s\int_{\T}\sigma^2 \diff x \diff t.
 \end{split}
    \end{equation}

In particular, returning to Equation \eqref{eq:reacede}, we may use Inequality \eqref{eq:rhsIII} to obtain the bound

\begin{equation}\label{eq:rhsIV}
        \int_{\T} |\sigma_s|^2 \diff x + \frac{1}{2}\int_0^s\int_{\T} |\partial_x\sigma|^2  \diff x \diff t \leqslant \int_{\T} |\sigma_0|^2  \diff x  + C\int_0^s\int_{\T} \sigma^2_t \diff x \diff t.
\end{equation}

From Inequality \eqref{eq:rhsIV}, it then follows from a Gr\"onwall argument that 

\begin{equation}
    \int_{\T} |\sigma_s|^2 \diff x \leqslant \exp(Cs)\int_{\T} |\sigma_0|^2  \diff x . 
\end{equation}

Moreover, under the assumption that $\sigma_0 \in L^\infty(\T) \subset L^2(\T)$, it is clear that $\sigma$ is bounded in $L^\infty([0,T];L^2(\T))$ and hence, it follows from Inequality \eqref{eq:rhsIV} that $\sigma$ is also bounded uniformly in $L^2([0,T];H^1(\T))\hookrightarrow L^2([0,T];BV(\T))$. 
In particular, this $L^2([0,T];H^1(\T))$ bound depends only on $W,V,F_i$ through the constant $C$ and on the initial data $\sigma_0$ through $\|\sigma_0\|_{L^2(\T)}$.

\medskip 

Having established the necessary $BV$ estimate for $\sigma$ we now turn to $f(r)$ and, once more, consider the time dissipation of the energy 
\begin{equation}\label{ener:react}
    \int_{\T}|\partial_x f(r) +V| \diff x.
 \end{equation} 
Using Equation \eqref{eq:reacf}, the Gr\"onwall argument follows exactly as in Theorem \ref{maintheorem} except, with the addition of the term 
\begin{equation}\label{eq:extra}
\begin{split}
    \int_0^s\int_{\T} \sgn (\partial_x f(r) +V)\partial_x\bigg(\tilde{F}_1(r,\sigma)- \tilde{F}_2(r,\sigma) \bigg) & \\
    \leqslant \sum_{i=1}^2 \bigg(\|\partial_\sigma\tilde{F_i}(r,\sigma)\|_{L^\infty([0,T]\times\T)}\|\partial_x\sigma\|_{L^1([0,T]\times\T)} + \|\partial_r\tilde{F_i}(r,\sigma)\|_{L^\infty([0,T]\times\T)}&\|\partial_x r\|_{L^1([0,T]\times\T)}\bigg).
\end{split}
\end{equation}
Since we have already established an a priori estimate for $\partial_x \sigma$ in $L^2([0,T]\times\T)$, it follows from Inequality \eqref{eq:extra} that the Gr\"onwall argument for Energy \eqref{ener:react} may be closed as long as $\partial_\sigma \tilde{F}_i$ and $\partial_r \tilde{F}_i$ are uniformly bounded.
These derivatives are certainly bounded when $\tilde{F}_i \in W^{1,\infty}([0,\infty)\times [0,1])$ and, moreover, the mapping 
\[
(r,\sigma) \mapsto (r\sigma, (1-r)\sigma)
\]
is locally Lipschitz continuous.
Since we assume that $F_i\in W^{1,\infty}([0,+\infty)^2)$, it then follows from Equation \eqref{eq:newreaction} that $\tilde{F}_i \in W^{1,\infty}([0,\infty)\times [0,1])$ if $\sigma, r \in L^\infty([0,T]\times\T)$. 

\medskip 

We already know that $r \in [0,1]$, meanwhile, it is assumed that $\sigma_0\in L^\infty(\T)$ from which it follows from \cite[Page 192, Chapter III, Theorem 8.1]{Ladyzhenskaya} that $\sigma\in L^\infty([0,T]\times \T)$. 
Consequently, the Gr\"onwall argument for Energy \eqref{ener:react} is closable when we assume $F_i\in W^{1,\infty}([0,+\infty)^2)$ and $\rho_{1,0},\rho_{2,0} \in L^\infty(\T)$.

\medskip

As with System \ref{eq:cdid}, given a suitable smooth approximation, the uniform $BV$ estimates on $\sigma$ and $f(r)$ are sufficient to provide strong compactness for the individual species and so it is left to consider the passage to the limiting system.
In particular, the reaction terms $F_i$ are assumed to be bounded and Lipschitz continuous and hence, if any sequences $(\rho_{1,n})_{n\ge 0},(\rho_{2,n})_{n\ge 0}$ would strongly converge in $L^p([0,T]\times \T)$ to $\rho_1,\rho_2$ respectively, it follows that the sequences $\left(F_i(\rho_{1,n}, \rho_{2,n})\right)_{n\ge 0}$, $i \in \{1,2\}$ would converge strongly in $L^p([0,T]\times \T)$ to $F_i(\rho_1,\rho_2)$, as $n\to+\infty$.
\end{proof}

\subsection*{Acknowledgements.}
We would like to thank Filippo Santambrogio for his pertinent comments and remarks, which helped us to improve the presentation of our results. We would like to also thank Jakub Skrzeczkowski for fruitful discussions regarding Kato's inequality.
The second author was supported by the Engineering and Physical Sciences Research Council [Grant Number EP/W524426/1].

\begin{itemize}
    \item Conflict of Interest - The authors declare that there are no known conflicts of interest associated with this
manuscript and no financial support has influenced its outcome.
\end{itemize}
\begin{appendices}
\fakesection{appendix}
\begin{lemma}\label{gradientinterchangelemma}
    Given $g\in W^{1,1}([0,T]\times \T)$, the following equality holds for every $s \in [0,T]$.
    \begin{equation}\label{eq:intchange}
    \int_{\T} |g| \diff x\bigg|_{t=s} = \int_{\T} |g|  \diff x\bigg|_{t=0} + \int_0^s\int_{\T} \sgn(g)\partial_t g \diff x \diff t.
    \end{equation}
    \end{lemma}
    \begin{proof}
    Since $g\in W^{1,1}([0,T]\times \T)$, the map $t\mapsto g(t, x)$  belongs to $W^{1,1}([0,T])$ for almost every $x \in \T$. 
    Moreover, since the absolute value function is Lipschitz continuous, it follows from the unpublished result of Serrin (according to the introduction of \cite{LeoniGiovanniMorini}, see also \cite[Theorem 2.1.11]{ziemer2012weakly}) that, for almost every $x \in \T$, the composition 
    \[[0,T]\ni t\mapsto |g(t,x)|\]
    is absolutely continuous and hence, for almost every $ x\in \T$ the map $t\mapsto |g(t,x)|$ is differentiable for almost every $t\in [0,T]$ - this result is sometimes known as Stampacchia's Theorem.
    Moreover, it also follows from this result that, at points of differentiability, the weak chain rule
    \begin{equation}\label{eq:chainrule}
    \partial_t|g(t,x)| = \sgn(g(t,x))\partial_tg(t,x)   
    \end{equation} 
    is satisfied.
    Furthermore, as a consequence of the absolute continuity of the map $t \mapsto |g(t,x)|$, the following equality is justified for every $s\in [0,T]$.
    \begin{equation}\label{eq:accontinuity}
    \int_{\T}|g|\diff x\bigg|_{t=s} - \int_{\T}|g|\diff x\bigg|_{t=0} = \int_{\T}\int_0^s \partial_t|g|\diff t \diff x = \int_{\T}\int_0^s \sgn(g)\partial_tg\diff t \diff x.
    \end{equation}
    To conclude Equation \eqref{eq:intchange} holds, we recognise that the order of integration may be swapped on the right-hand side of Equation \eqref{eq:accontinuity}. 
    This is justified by Fubini's Theorem since $\partial_t |g| \in L^1([0,T]\times\T)$.
    \end{proof}
    \begin{lemma}\label{modulusinterchangelemma}
        Given $g,h \in C([0,T];W^{1,1}(\T))$, the following equality holds almost everywhere on $[0,T]\times\T$
        \begin{equation}\label{eq:moduluschange}
            \partial_x(|g|h) = \sgn(g)\partial_x(gh).
        \end{equation}
    \end{lemma}
    \begin{proof}
        As discussed in Lemma \ref{gradientinterchangelemma}, since $g \in C([0,T];W^{1,1}(\T))$, it follows from the unpublished work of Serrin (the Stampacchia Theorem), that the map $\T \ni x\mapsto |g(t,x)|$ is absolutely continuous for every $t\in [0,T]$. 
        Moreover, it satisfies the weak chain rule 
        \[
        \partial_x|g(t,x)| = \sgn(g(t,x))\partial_xg(t,x)
        \]
        at its points of differentiability. 
        Similarly, since the map $x \mapsto h(t,x)$ is also absolutely continuous for every $t\in [0,T]$, it follows that $h$ and moreover the product $gh$ is differentiable in $x$ almost everywhere on $[0,T]\times\T$.
        In particular, at the points of differentiability, the following equality is satisfied as consequence of the product rule.
        \[
        \partial_x(|g|h) = |g|\partial_xh + \sgn(g)\partial_x g h = \sgn(g)(g\partial_x h + h\partial_x g) = \sgn(g)\partial_x(gh).\]
    \end{proof}
    \begin{lemma}
        If $u \in L^1(\T)$ and $\partial_{xx}^2 u \in L^1(\T)$, then $\partial_{xx}^2 |u| \in \mathscr{M}(\T)$. 
        Moreover, the inequality 
        \begin{equation}\label{eq:Kato}
            \partial_{xx}^2 |u| \geqslant \sgn(u)  \partial_{xx}^2(u)
        \end{equation}
        is satisfied in the sense of distributions.
    \end{lemma}
    \begin{proof}
    This is the consequence of the classical Kato's inequality. The original proof of Kato's inequality is established by Kato in \cite[Lemma A]{Kato1972} (see also \cite{brenier}) and states that, if $\partial_{xx}^2 u \in L^1_{loc}(\R)$, then $\partial_{xx}^2 |u| \in \mathscr{M}_{loc}(\R)$ whilst Inequality \eqref{eq:Kato} holds in the sense of distributions. 
        Moreover, since this result holds locally, it readily generalises to $\T$, since any subset of $\T$ may be viewed as a compact subset of $\R$.
    \end{proof}
\end{appendices}

\printbibliography

\end{document}